\documentclass[11pt,reqno]{amsart}
\usepackage{amsmath}
\usepackage{authblk} 

\usepackage{algorithm}
\usepackage{algpseudocode}
\usepackage[utf8]{inputenc}

\usepackage{bm}
\usepackage{graphicx}
\usepackage{caption}
\usepackage{subcaption}
\usepackage[margin = 3cm]{geometry}
\usepackage{enumerate} 
\usepackage{parskip}
\usepackage{hyperref}
\usepackage{comment}

 \usepackage{xcolor}
\usepackage{amsmath}
\usepackage{booktabs}
\usepackage{siunitx}
\usepackage{enumitem}
\usepackage{amsmath,amssymb}
\usepackage{algorithm}
\usepackage{algpseudocode}
\usepackage{caption}
\usepackage[foot]{amsaddr}
\allowdisplaybreaks

\def\tr{\textcolor{red}}

\newtheorem{theorem}{Theorem}[section]
\newtheorem{lemma}[theorem]{Lemma}
\newtheorem{corollary}[theorem]{Corollary}
\newtheorem{proposition}[theorem]{Proposition}
\theoremstyle{definition}
\newtheorem{assumption}[theorem]{Assumption}

\newtheorem{example}[theorem]{Example}

\theoremstyle{remark}
\newtheorem{remark}[theorem]{Remark}
\numberwithin{equation}{section}

\numberwithin{equation}{section}
\renewcommand{\theequation}{\arabic{section}.\arabic{equation}}
\newcommand{\red}[1]{\textcolor{red}{#1}}
\newcommand{\blue}[1]{\textcolor{blue}{#1}}

\makeatletter
\renewcommand{\email}[2][]{%
  \ifx\emails\@empty\relax\else{\g@addto@macro\emails{,\space}}\fi%
  \@ifnotempty{#1}{\g@addto@macro\emails{\textrm{(#1)}\space}}%
  \g@addto@macro\emails{#2}%
}
\begin{document}

\makeatother

\makeatletter
\newenvironment{breakablealgorithm}
{%
  \par
  \noindent
  \refstepcounter{algorithm}%
  \hrule height .8pt depth 0pt
  \kern -3pt 
  \renewcommand{\caption}[2][\relax]{%
    \par\noindent
    {\textbf{\ALG@name~\thealgorithm} ##2\par}%
    \ifx\relax##1\relax
      \addcontentsline{loa}{algorithm}
      {\protect\numberline{\thealgorithm}##2}%
    \else
      \addcontentsline{loa}{algorithm}
      {\protect\numberline{\thealgorithm}##1}%
    \fi
    \kern 2pt
    \hrule
    \kern 2pt
  }%
}
{%
  \kern 2pt
  \hrule
  \par
}
\makeatother

\title[Variational Neural Transport] 
{Variational Neural Transport for Aggregation-Diffusion Equations}


\author{Qianyi Li$^{\dagger}$}
\address{$^{\dagger}$Department of Mathematics, Iowa State University, Ames, IA 50011} 
\email{liqiany@iastate.edu}


\author{Hailiang Liu$^\dagger$}
\address{$^\dagger$Department of Mathematics, Iowa State University, Ames, IA 50011}
\email[Corresponding author]{hliu@iastate.edu}

\date{\today}
\begin{abstract}
We develop a neural transport framework for nonlinear, nonlocal aggregation-diffusion equations arising in chemotaxis, polymer dynamics, and related applications. Based on the Onsager variational principle, 
we reformulate the governing equation as a constrained optimal control problem, taking the velocity field as the control and evolving the density via the associated continuity equation. The velocity field is parameterized by a neural network trained on the variational objective evaluated along particle trajectories. By coupling the density and generalized score equations into a single system solved via the Euler method, the score function propagates efficiently for loss evaluation.
Furthermore, we establish a rigorous error analysis that decomposes the total error into approximation, optimization, Monte Carlo sampling, and time-discretization components, and prove corresponding error bounds. Numerical experiments across several representative models 
demonstrate the accuracy and effectiveness of the proposed approach. 
\end{abstract}

\maketitle

\medskip

\section{Introduction}

This paper concerns with the aggregation-diffusion equation with no-flux boundary conditions,
\begin{equation}\label{Ori_sys}
    \begin{cases}
         \partial_t \rho = \nabla_x \cdot [\rho \nabla_x(H'(\rho) + V(x)+ W \ast \rho)], \quad & x \in \Omega \subset \mathbb{R}^d, \quad t > 0\\
        \rho(x,0) = \rho_0(x), & x \in \Omega \subset \mathbb{R}^d,\\
        \rho \nabla_x(H'(\rho) + V(x)+ W \ast \rho)\cdot \bm{\mathrm{n}} = 0,
    & x \in \partial \Omega,\quad t>0,
    \end{cases}
\end{equation}
where $\Omega$ is a bounded domain in $\mathbb{R}^d$, $\rho = \rho(t,x)$ denotes the unknown density, $H(\rho)$ is the internal energy density, $V(x)$ is a prescribed confinement potential, and $W(x)$ is an interaction potential, and $\mathbf{n}$ denotes the outward unit normal to the domain boundary $\partial\Omega$. Under suitable assumptions on $H$, $V$ and $W$, the existence of weak solutions to (\ref{Ori_sys}) is well-established; see, for example, \cite{CMV03}.
{Since the equation conserves the total mass \(M=\int_\Omega \rho\,dx\), which need not equal \(1\); a probability density corresponds specifically to the normalized case \(M=1\).}

Aggregation-diffusion equations of the form \eqref{Ori_sys} arise in a wide range of applications. Two standard choices for the internal energy are  $H(\rho)=\rho\log \rho -\rho$ and  $H(\rho)=\rho^m$, corresponding to linear diffusion and porous-medium-type diffusion, respectively. In the absence of the confinement and interaction potentials, the equation reduces to the heat equation or porous medium equation \cite{PME_JL}. The addition of a confinement potential leads to a mass-conserving Fokker--Planck equation
\cite{FPalma991219920902101}, while the inclusion of an interaction potential gives rise to a broad class of nonlinear and nonlocal Fokker--Planck equations.  

Such models arise in diverse applications, including chemotaxis \cite{CHEperthame2007transport}, biological aggregation \cite{AGG1grunbaum_modelling_1994, AGG2kawasaki1978diffusion, AGG3topaz_nonlocal_2006}, and the dynamics of rod-like polymers \cite{ROGdoi1988theory}. Notable examples include the Patlak-Keller-Segel system \cite{PKS1keller_initiation_1970, PKS2Patlak1953311} and the Doi-Onsager equation \cite{ROGdoi1988theory, DOI2Constantin2004365, DOI3liu2005axial, DOI3}. These examples illustrate the broad scope of aggregation-diffusion models and their ability to describe the interplay between diffusion, external confinement, and nonlocal interactions.


The aggregation-diffusion equation possesses several fundamental structural properties. In particular, it conserves the total mass, preserves the nonnegativity of the density, and admits the free-energy functional \cite{CMV03,ENE2,ENE3Villani2003}:
\begin{equation}\label{Free_energy_E}
    E(\rho) = \int_\Omega H(\rho)\,dx + \int_\Omega V(x)\rho \,dx + \frac{1}{2}\int_\Omega \int_\Omega W(x - y)\rho(y)\rho(x)\,dy\,dx,
\end{equation}
which consists of the internal, potential, and interaction energies. 
Along sufficiently regular solutions, the free energy satisfies the dissipation identity
\begin{equation*}
    \frac{d}{dt}E(\rho) =  - \int_\Omega \rho|\nabla_x(H'(\rho) + V(x) + W \ast \rho)|^2 \,dx \leq 0.
\end{equation*}
Thus, the dynamics can be viewed as a dissipative evolution driven by the free-energy functional.
Preserving these intrinsic structural properties at the discrete level is important for accurately capturing the physical behavior and long-time dynamics of the system. Numerous structure-preserving numerical methods have therefore been developed to preserve, either exactly or approximately, properties such as mass conservation, positivity, and energy dissipation; see, for example, \cite{CLASSD1, CLASSD2liu_entropy_2012, CLASSD3liu2014maximum, CLASSD4liu_entropy_2015, CLASSD5liu_free_2014, CLASSD6liu_free_2017, CLASSD8sun_discontinuous_2018, CLASSD9, CLASSD10liu_positive_2020}. Despite their success, most of these approaches rely on mesh-based spatial discretizations. The resulting computational cost and memory requirements can grow rapidly with the spatial dimension, making such methods challenging to apply to high-dimensional problems. 

Neural-network-based methods have recently emerged as a promising alternative for solving high-dimensional PDEs; see, for example,
\cite{Ritz18, DGM18, HDDL18, WCC20, LuNO19, FBPINN23,WAN20, LH24}. Rather than constructing a mesh or explicitly discretizing the spatial domain, these methods approximate the solution or relevant solution components directly through neural-network parameterizations. This mesh-free representation can substantially reduce the dependence of computational complexity on spatial dimension and provides an attractive framework for high-dimensional PDE problems.

However, standard neural PDE solvers generally do not automatically preserve the intrinsic structures of density-evolution equations, including positivity, mass conservation, and free-energy dissipation. This limitation can be particularly significant for long-time simulations, where small violations of these properties may accumulate and lead to qualitatively incorrect dynamics. Consequently, increasing attention has been devoted to neural methods
that incorporate the underlying probabilistic, geometric, or variational
structures of the governing equations into the approximation and training
procedure.

In contrast, we adopt a variational approach based on  the Onsager principle. Rather than directly approximating the density, we formulate the dynamics in terms of a velocity field and its associated  continuity equation. This formulation naturally preserves both mass conservation and density positivity while yielding a principled variational objective to learn the transport velocity via a neural network.

Recently, several complementary neural approaches have been developed to solve Fokker--Planck and related equations. Self-consistency methods characterize transport velocity through a fixed-point relation between the velocity field and its generated probability flow, creating neural objectives to  approximate density evolution \cite{shen22a}. This framework was extended via self-consistent velocity matching, which applies to broad mass-conserving PDEs, including Wasserstein gradient flows, by using iterative optimization and biased gradient estimators to enhance efficiency \cite{LHS23}. Parallel  work reformulates the Fokker--Planck equation as a deterministic probability flow and learns the score function on-the-fly from transported samples, providing a mesh-free representation of high-dimensional density evolution with KL-divergence guarantees \cite{Eric23}. This score-based approach was further generalized to mean-field Fokker--Planck equations with density-dependent interactions, complete rigorous error analysis \cite{lu23}. Our formulation distinctively advances the existing methods in two key aspects: 
\begin{itemize} 
\item Distinction from Self-Consistency Methods  \cite{LHS23, shen22a}: Instead of fixed point relations, our approach  derives directly from the Onsager variational
principle. This yields   a variational objective whose global minimizer recovers the original aggregation--diffusion equation while retaining
its fundamental energy-dissipation structure. 
\item Distinction from  Score-Based Methods \cite{Eric23,lu23}: Our method introduces a generalized score, extending particle-based score evolution beyond standard 
linear diffusion to broader internal energies, such as 
porous-medium-type dynamics.
\end{itemize}

While \cite{zhou2025} formulates Fokker--Planck evolution as a mean-field control problem via score-based normalizing flows,
we extend this formulation to general  aggregation-diffusion equations, featuring  nonlinear diffusion, external confinement, and nonlocal interactions. Rigorously grounded via the Onsager variational principle,  our approach establishes equivalence to the target PDE and guarantees KL-stability 
in the linear diffusion regime. 
Crucially, we present the first unified error analysis for the proposed methodology, explicitly bounding errors from neural approximation, optimization, Monte Carlo sampling, and time discretization.

The main contributions of this work are summarized as follows.
\begin{itemize}
    \item \textbf{Structure-Preserving Variational Formulation:} We establish an Onsager-based  variational formulation using transport velocity that inherently guarantees mass conservation, density positivity, and free-energy dissipation structure via a constrained optimization problem.   
    \item \textbf{Particle-Based Neural Algorithm:} We parameterize the velocity field using a neural network and evaluate loss along trajectories, using a coupled score-type equation, to avoid explicitly solving high-dimensional density equations.

    \item \textbf{Theoretical Foundation \& Stability:} We prove the  equivalence between the variational formulation and the target equation. Furthermore, we establish stability and relative entropy estimates.
    \item \textbf{Unified Error Analysis:} 
   We provide convergence guarantees by decomposing total error into approximation, optimization, time discretization, and Monte Carlo sampling
   components. 
\end{itemize}

The remainder of this paper is organized as follows. Section~2 derives the constrained optimization formulation for aggregation--diffusion equations using Onsager's variational principle. Section~3 establishes its theoretical foundations, including the equivalence result, KL divergence-based stability estimates, and trajectory-wise dynamics of density and generalized scores. Section~4 introduces the particle-based neural algorithm, where a neural network 
parameterizes the velocity field and the variational objective is evaluated along characteristic trajectories. Section~5 provides a rigorous error analysis quantifying neural approximation, optimization, temporal discretization, and Monte Carlo sampling errors. Section~6 demonstrates  numerical experiments across several representative models and validates the error bounds from  Section ~5. Technical proofs are deferred to Appendix~ \ref{AppendixA}. 

\paragraph{\bf Notation and Function Spaces.} Throughout this paper, $\Omega \subset \mathbb{R}^d$ denotes an open domain. For $k \in \mathbb{N}_0$ and $1 \le p \le \infty$, $W^{k,p}(\Omega)$ denotes the standard Sobolev space of functions in $L^p(\Omega)$ whose weak derivatives up to order $k$ belong to $L^p(\Omega)$. For $1 \le p < \infty$, the Sobolev norm is given by
\[
\| f \|_{W^{k,p}(\Omega)} := \left( \sum_{|\alpha| \le k} \| \partial^\alpha f \|_{L^p(\Omega)}^p \right)^{1/p},
\]
where $\alpha = (\alpha_1, \dots, \alpha_d) \in \mathbb{N}_0^d$ is a multi-index with length $|\alpha| = \sum_{i=1}^d \alpha_i$. For $p = \infty$, the norm is defined as $\| f \|_{W^{k,\infty}(\Omega)} := \max_{|\alpha| \le k} \| \partial^\alpha f \|_{L^\infty(\Omega)}$. For vector fields $v \in W^{k,p}(\Omega; \mathbb{R}^d)$, the norm is understood componentwise as $\| v \|_{W^{k,p}(\Omega;\mathbb{R}^d)} := (\sum_{i=1}^d \| v_i \|_{W^{k,p}(\Omega)}^p)^{1/p}$.

\section{Variational Formulation}

We consider the nonlinear aggregation-diffusion equation (\ref{Ori_sys}).  Assume that
\begin{itemize}
    \item $H\in C^2((0,\infty))$ is strictly convex and superlinear;
    \item $V,W\in C^2(\Omega)$, $W(x)=W(-x)$, and both $V$ and $W$ are semiconvex and bounded from below;
    \item the initial data satisfies $\rho_0\ge0$, has finite mass, and finite free energy.
\end{itemize}

Then the free energy
\begin{equation}
E(\rho)=\int_{\Omega} H(\rho)\,dx+\int_{\Omega} V(x)\rho(x)\,dx+\frac12\int_{\Omega}\int_{\Omega}W(x-y)\rho(x)\rho(y)\,dx\,dy
\end{equation}
generates a unique global Wasserstein gradient flow solving \eqref{Ori_sys} (see, e.g., \cite{CMV03,AGS}). Moreover, the solution preserves mass and satisfies the
energy dissipation identity
\begin{equation}
E(\rho(t))+\int_0^t\int_{\Omega}\rho\left|\nabla_x \mu[\rho] \right|^2\,dx\,ds=E(\rho_0),\qquad t\ge0,
\end{equation}
where the notation
\begin{equation}\label{mu_rho}
    \mu[\rho]: = H'(\rho)+V+W*\rho.  
\end{equation}
will also be used in the following analysis.
Moreover, when $t \rightarrow \infty$, the solution approaches a stationary equilibrium. More precisely, we have the following proposition.

\begin{proposition}\label{KL_sta}
Let $\rho$ be a solution to (\ref{Ori_sys}) and $\pi$ be a stationary solution satisfying
\begin{equation*}
\mu[\pi]={\rm const}. 
\end{equation*}
Define the relative free energy  
$$
E(\rho \mid \pi)=\int_{\Omega} (H(\rho)-H(\pi) -H'(\pi)(\rho-\pi))dx +\frac{1}{2}\int_{\Omega \times \Omega} 
W(x-y)(\rho-\pi)(x)(\rho-\pi)(y)dxdy. 
$$
If $W$ is positive semi-definite, then $E(\rho\mid \pi) \geq 0$ and is monotonically non-increasing along solutions of  (\ref{Ori_sys}). 
More precisely,  
$$
\frac{d}{dt}E(\rho\mid\pi)=-\int_{\Omega} \rho |\nabla_x \mu [\rho]|^2 dx.  
$$
Consequently, the relative free energy is a Lyapunov functional for \eqref{Ori_sys}.
\end{proposition}

A detailed proof is given in Appendix~\ref{AppendixA}.

\begin{remark}[Long-time convergence]
Under the additional assumptions that $H$ satisfies the displacement
convexity condition, $V$ is uniformly convex with
\[
    \nabla_x^2 V \geq \lambda I
\]
for some $\lambda>0$, and $W$ is convex, the relative free energy
converges exponentially to equilibrium:
\[
    E(\rho\mid\pi)
    \leq e^{-2\lambda t} E(\rho_0\mid\pi).
\]
This quantitative convergence result follows from
\cite[Theorem~2.1]{CMV03}.
It provides theoretical support for the long-time behavior observed
in the numerical experiments of Section~5.
\end{remark}
\begin{remark} The first term in $E(\rho\mid \pi)$
reduces to Kullback–Leibler divergence
$$
D_{KL}(\rho\mid \pi):=\int_{\Omega} \rho \log \frac{\rho}{\pi}dx 
$$
if $H(\rho)=\rho\log \rho -\rho$ (entropy case); and to the Bregman divergence
$$
\int_{\Omega} (\rho^m-\pi^m-m\pi^{m-1}(\rho-\pi))dx  
$$
if $H(\rho)=\rho^m$ (porous-medium case).
\end{remark}


We now turn to deriving a variational formulation of (\ref{Ori_sys}). To do so, we introduce a velocity field $v(x,t)$ and rewrite the evolution equation as the continuity equation
\begin{equation}\label{Model_eq}
        \partial_t \rho + \nabla_x \cdot (\rho v) = 0,
\end{equation}
together with the no-flux boundary condition
\begin{equation}\label{Model_bc}
    (\rho v) \cdot \bm{\mathrm{n}} = 0, \quad \text{for} \quad x \in \partial \Omega.
\end{equation}
The free energy functional $E(\rho)$ remains defined in (\ref{Free_energy_E}), with corresponding chemical potential defined in (\ref{mu_rho}).

Following Onsager's variational principle as shown in \cite{CLX04}, we define the Rayleighian as the sum of the free energy change rate and the kinetic dissipation,
\begin{equation*}
    \begin{aligned}
        R(v,\rho) & = \dot{E} + \frac{1}{2}\int_\Omega \rho |v|^2 \,dx.
    \end{aligned}
\end{equation*}
Using the continuity equation, $\frac{\delta E}{\delta \rho} = \mu[\rho]$, and integration by parts
$$
\dot E= \int_\Omega \frac{\delta E}{\delta \rho} \cdot \partial_t \rho \,dx = \int_{\Omega} \rho v \cdot \nabla_x \mu dx,  
$$
so that 
\begin{equation*}
    \begin{aligned}
        R(v,\rho) = \frac{1}{2}\int_\Omega \rho |v + \nabla_x\mu[\rho]|^2\,dx - \frac{1}{2}\int_\Omega \rho |\nabla_x\mu[\rho]|^2 \,dx.
    \end{aligned}
\end{equation*}
Since the second term is independent of $v$, Onsager's principle states that the physical velocity minimizes 
\begin{equation*}
    \frac{1}{2}\int_\Omega \rho |v + \nabla_x\mu[\rho]|^2\,dx.
\end{equation*}

Integrating this dissipation over the time interval $[0, T]$ and imposing the continuity equation  (\ref{Model_eq}) as a constraint leads naturally to the following optimal control formulation:
\begin{equation}\label{opt0}
    \begin{aligned}
        & \min_{v,\rho} J(\rho, v)\\
        &\text{s.t.} \quad   \partial_t \rho + \nabla_x \cdot (\rho v) = 0,\\
        &\text{BC:} \quad  (\rho v) \cdot \bm{\mathrm{n}} = 0 \quad \text{for} \quad x \in \partial \Omega, \\
        & \text{IC}:\quad \rho(0, \cdot) = \rho_0,
    \end{aligned}
\end{equation}
where
\begin{equation*}
   J(\rho, v)
   :=
   \frac{1}{2}\int_0^T\int_\Omega
   \rho |v + \nabla_x\mu[\rho]|^2\,dx\,dt.
\end{equation*}

Here, the density $\rho$ serves as the state variable, the velocity field $v$ acts as the control, and the continuity equation enforces mass conservation. The resulting optimization problem offers a variational characterization of the aggregation–diffusion dynamics, with the minimizer formally recovering the PDE evolution. 

\section{Theoretical foundations}
\subsection{Equivalence} We first establish the equivalence between the constrained optimization problem (\ref{opt0}) and problem  (\ref{Ori_sys}).

{
\begin{theorem}\label{E_U}
Assume that the nonlinear aggregation--diffusion equation (\ref{Ori_sys}) admits a unique solution $\rho^*$. Then the constrained optimization problem (\ref{opt0}) attains its global  minimum value zero at $(\rho^*,v^*)$, where
\[
v^*=-\nabla_x \bigl(H'(\rho^*)+V+W*\rho^*\bigr).
\]
Conversely, any minimizer $(\rho,v)$ of (\ref{opt0}) determines a solution $\rho$ of (\ref{Ori_sys}). Hence, the constrained optimization problem (\ref{opt0}) is equivalent to the aggregation--diffusion problem (\ref{Ori_sys}). Moreover, the density $\rho^*$ is uniquely determined, and the optimal velocity $v^*$ is uniquely determined in $L^2_{\rho^*}((0, T)\times \Omega)$.
\end{theorem}

\begin{proof}
The proof proceeds in two steps.

\noindent\textbf{Step 1.} 
Let $\rho^*$ be a solution to (\ref{Ori_sys}). Upon rewriting, we have
\begin{equation*}
    \partial_t \rho^*+\nabla_x\cdot(\rho^*v^*)=0,
\end{equation*}
where
\begin{equation*}
    v^*
    =-\nabla_x H'(\rho^*)-\nabla_x V-\nabla_x(W\ast\rho^*)
    =-\nabla_x\mu[\rho^*].
\end{equation*}
Also, the no-flux boundary condition in (\ref{Ori_sys}) reduces to
\begin{equation*}
    (\rho^*v^*)\cdot\bm{\mathrm n}=0.
\end{equation*}
Thus, $(\rho^*,v^*)$ satisfies the constraints in (\ref{opt0}). Since the objective functional is nonnegative and
\[
J(\rho^*,v^*)=0,
\]
$(\rho^*,v^*)$ is a global minimizer.

\noindent\textbf{Step 2.}
Let $(\rho,v)$ be any minimizer of (\ref{opt0}). Since the global minimum value of $J$ is zero, we have $
J(\rho,v)=0$, which implies 
\[
\rho\left|v+\nabla_x\mu[\rho]\right|^2=0
\qquad\text{a.e. in }(0,T)\times\Omega.
\]
Define the momentum vector field
$m:=\rho v$. The pointwise identity above yields:
\[
m=-\rho\nabla_x\mu[\rho]
\qquad\text{a.e. in }(0,T)\times\Omega,
\]
where both sides naturally vanish on the set $\{(t, x): \quad \rho(t, x)=0\}$. Substituting this  momentum identity into the continuity equation $\partial_t +\nabla_x\cdot m=0$ gives
\[
\partial_t\rho
=
\nabla_x\cdot
\left(
\rho\nabla_x
\bigl(
H'(\rho)+V+W*\rho
\bigr)
\right),
\]
complemented by the no-flux boundary condition. Thus, $\rho$ solves (\ref{Ori_sys}). By the assumed uniqueness of the solution to (\ref{Ori_sys}), we obtain $
\rho=\rho^*.$ Consequently,
\[
m
=-\rho^*\nabla_x\mu[\rho^*]
=\rho^*v^*
=:m^*.
\]
Therefore, the density and momentum are uniquely determined. 
\end{proof}
} 

\begin{corollary}[Energy dissipation]
\label{cor:energy_dissipation}
Let $(\rho^*,v^*)$ be the global minimizer characterized in
Theorem~\ref{E_U}. Then
\[
\frac{d}{dt}E(\rho^*)
=
-\int_\Omega
\rho^*|\nabla_x\mu[\rho^*]|^2\,dx
=
-\int_\Omega
\rho^*|v^*|^2\,dx
\le 0.
\]
Hence, the global minimizer recovers the free-energy dissipation
structure of the aggregation--diffusion equation.
\end{corollary}

\begin{proof}
The free-energy evolution can be written as
\[
\frac{d}{dt}E(\rho)
=
\frac{1}{2}\int_\Omega
\rho|v+\nabla_x\mu[\rho]|^2\,dx
-\frac{1}{2}\int_\Omega
\rho|v|^2\,dx
-\frac{1}{2}\int_\Omega
\rho|\nabla_x\mu[\rho]|^2\,dx.
\]
By Theorem~\ref{E_U},
\[
v^*+\nabla_x\mu[\rho^*]=0
\qquad \rho^*\text{-a.e.},
\]
so that
\[
\frac{d}{dt}E(\rho^*)
=
-\frac{1}{2}\int_\Omega
\rho^*|v^*|^2\,dx
-\frac{1}{2}\int_\Omega
\rho^*|\nabla_x\mu[\rho^*]|^2\,dx.
\]
Since $v^*=-\nabla_x\mu[\rho^*]$, the conclusion follows.
\end{proof}

\subsection{Control of the KL Divergence}
   
    We next provide a theoretical justification for the variational formulation (\ref{opt0}) by relating the error in the velocity field to the discrepancy between the corresponding density distributions. Throughout this paper, we consider two choices of the internal energy 
    $$H(\rho)=\rho\log\rho - \rho  \; \text{and}\;  \rho^{m}.
    $$
    For the logarithmic entropy $H(\rho) = \rho\log \rho - \rho$, the natural measure of  discrepancy, is the Kullback--Leibler(KL) divergence:
    \begin{equation}
        D_{KL}(\rho\|\rho^*) = \int_\Omega \rho(t,x)\log\frac{\rho(t,x)}{\rho^*(t,x)}\,dx,
    \label{eq:DKL}
    \end{equation}
    where $\rho^*$ denotes the reference solution.


    
    

\begin{lemma}\label{L2_KL}
    If $0 < m^* \leq  \rho_1,\rho_2 \leq M^*$, and $\int_\Omega \rho_1\,dx = \int_\Omega \rho_2\,dx$, there exists a constant $C_*$, depending on $m^*$ and $M^*$, such that
    \begin{equation*}
        \|\rho_1 - \rho_2\|_{L^2(\Omega)}^2 \leq  C_* D_{KL}(\rho_1\|\rho_2).
    \end{equation*}
\end{lemma}
The following lemma characterizes the time evolution of the KL divergence.
\begin{lemma}\label{Lemma_KL_t}
    Let $\rho^*$ be the solutions to (\ref{Ori_sys}),
    and $\rho$ be the solution to the model equations (\ref{Model_eq}) together with its boundary condition (\ref{Model_bc}). Then the time derivative of the KL divergence (\ref{eq:DKL}) is given by:
    \begin{equation}\label{DKL_t}
\begin{aligned}
    \frac{d}{dt}D_{KL}(\rho \| \rho^*)  & = \int_\Omega  \rho\nabla_x\log \frac{\rho}{\rho^*} \cdot  (v - v^*) \,dx,
\end{aligned}
\end{equation}
where $v^* = - \nabla_x \log \rho^* - \nabla_x V - \nabla_x( W \ast \rho^*)$,
\end{lemma}

Combining these lemmas, we obtain the following estimate on the KL divergence:

\begin{theorem}\label{KL_control}
    (Time evolution of KL divergence) Let $\rho^*$ be the solution to (\ref{Ori_sys}) with initial density 
    $$
    0<m_0\leq \rho_0 \leq M_0. 
    $$
    Let $v$ be Lipchitz in $x$ and continuous in $t$,  and $\rho$ satisfies  equation (\ref{Model_eq}) subject to  boundary condition (\ref{Model_bc}) and same initial data $\rho_0$, then
    \begin{equation*}
        \frac{d}{dt}D_{KL}(\rho\|\rho^*)\leq \int_\Omega \rho| v + \nabla_x \log \rho + \nabla_x V + \nabla_x( W \ast \rho)|^2 \,dx +  C D_{KL}(\rho\|\rho^*).
    \end{equation*}
    Moreover, 
    \begin{equation}\label{Up_KL}
D_{KL}(\rho\|\rho^*)
\le
e^{CT}
\int_0^T
\left(
\int_\Omega \rho| v + \nabla_x \log \rho + \nabla_x V + \nabla_x( W \ast \rho)|^2 \,dx
\right) dt.
\end{equation}
where $C$ depends on $m_0, M_0, T, \|\nabla_x \cdot v\|_{L^\infty([0, T]\times \Omega)}$ and $\|\nabla_x W\|_{L^1(\Omega)}$.
\end{theorem}
The proof details are provided in Appendix~\ref{AppendixA}.

\begin{remark}[Regularization of non-negative initial data] If the initial density $\rho_0$ vanishes at some points, the strict positivity assumption
in Theorem~\ref{KL_control} is not satisfied. In the numerical
implementation, we therefore regularize the initial density by adding a small
positive constant $\delta>0$ and renormalizing to preserve the
total mass. Specifically, we set
\[
    \rho_0^\delta(x)
    :=
    \frac{\rho_0(x)+\delta}
    {\displaystyle\int_\Omega (\rho_0(y)+\delta)\,dy}
    \int_\Omega \rho_0(y)\,dy.
\]
Then $\rho_0^\delta>0$ and 
$
\int_{\Omega} \rho_0^\delta(x)dx=\int_\Omega \rho_0(x)dx. 
$
In the numerical experiments, $\delta$ is chosen sufficiently small
relative to the discretization parameters so that the regularization
error is negligible compared with the discretization error.
\end{remark}

\begin{remark}[Porous-medium case] 
For the porous-medium case \(H(\rho)=\rho^m\), the KL-based argument above  does not directly apply to this case. A stability estimate in this setting would require an appropriate discrepancy measure and a separate analysis.  
\end{remark}

To learn the velocity field $v$,  we need access to the quantities $\nabla_x \log \rho$ and $W*\rho$, which encode the local score information and the nonlocal interaction effects, respectively. We next establish how these quantities can be propagated along particle trajectories, enabling their evaluation through the probability flow.

\subsection{Propagation along the particle path}

For general $H(\rho)$, we consider a probability flow generated by the ODE
\begin{equation}\label{ODE_x}
    \dot{x}_t = v(t,x_t),
\end{equation}
with random initialization $x_0\sim\rho_0$. The density
$\rho(t,\cdot)$ is the pushforward of the initial density $\rho_0$
under the flow map generated by \eqref{ODE_x}.
The continuity equation can be written as
\begin{equation}\label{material_rho}
    \partial_t\rho
    +
    v\cdot\nabla_x\rho
    =
    -\rho\nabla_x\cdot v.
\end{equation}
Therefore,
\begin{equation*}
\begin{aligned}
\partial_t H'(\rho)
+
v\cdot\nabla_x H'(\rho)
=
-\rho H''(\rho)\nabla_x\cdot v.
\end{aligned}
\end{equation*}
Set $
q:=\nabla_x H'(\rho)$, 
and take the gradient with respect to $x$ to  obtain
\begin{equation}\label{general_H}
\begin{aligned}
\partial_t q +
(\nabla_x v)^\top q 
+ v\cdot\nabla_x q =-\nabla_x\left(
\rho H''(\rho)\nabla_x\cdot v
\right).
\end{aligned}
\end{equation}
The right-hand side of \eqref{general_H} can be rewritten as 
$$
- \rho H''(\rho)\nabla_x(\nabla_x\cdot v)-
(\nabla_x\cdot v)
\left(
1+\frac{\rho H'''(\rho)}{H''(\rho)}
\right)
\nabla_x H'(\rho). 
$$
Hence along the particle trajectory $x_t$, we obtain 
\begin{equation}\label{general_rho}
    \dot \rho_t= -\rho_t \nabla_x\cdot v_t,
\end{equation}
and 
\begin{equation}\label{general_q}
    \dot q_t =-(\nabla_x v_t)^\top q_t
-\rho_t H''(\rho_t)\nabla_x(\nabla_x\cdot v_t) -
(\nabla_x\cdot v_t)
\left(
1+\frac{\rho_t H'''(\rho_t)}{H''(\rho_t)}
\right)q_t,
\end{equation}

where we have used notation
\begin{equation*}
    \rho_t:=\rho(t,x_t),
    \qquad
    v_t:=v(t,x_t), \quad q_t=q(t, x_t).
\end{equation*}
Moreover, the spatial derivatives of $v_t$ are understood as the
corresponding spatial derivatives of $v(t,x)$ evaluated at $x=x_t$.
\begin{remark}
For the entropy $H(\rho)=\rho\log\rho-\rho,$
the generalized score function $q= \nabla_x H'(\rho),$
reduces to the actual score function $s =\nabla_x\log\rho$.
Along the particle trajectory, we obtain the following ODE:
\begin{equation}\label{ODE_s}
    \dot{s}_t
    =
    -(\nabla_x v_t)^\top s_t
    -
    \nabla_x(\nabla_x\cdot v_t).
\end{equation}

For the porous-medium energy
\begin{equation*}
    H(\rho)=\rho^m,
    \qquad m>1,
\end{equation*}
It follows from \eqref{general_q}, along the particle trajectory $q_t = m\nabla_x\rho^{m-1}(t,x_t)$, together with \eqref{material_rho}, we have the following ODE system:
\begin{equation}\label{ODE_q}
\begin{cases}
\displaystyle
    \dot{\rho}_t
    =
    -\rho_t\nabla_x\cdot v_t,\\
\displaystyle
    \dot{q}_t
    =
    -(\nabla_x v_t)^\top q_t
    -(m-1)(\nabla_x\cdot v_t)q_t
    -
    m(m-1)\rho_t^{m-1}
    \nabla_x(\nabla_x\cdot v_t).
\end{cases}
\end{equation}
    
\end{remark}

\section{Neural/particle Discretization and Algorithm}

Before proceeding, we introduce the notation used throughout this section. The time interval $[0,T]$ is partitioned into $N_t$ uniform subintervals of size  $\Delta t = T/N_t$, with grid points 
$\{t_j\}_{j=0}^{N_t}$. 
Let $\{x_0^{(n)}\}_{n=1}^{N_p}$ be i.i.d. samples drawn from the initial density $\rho_0$. The corresponding particle trajectories, generated by the characteristic flow, are denoted by
\begin{equation*}
x^{(n)}(t_j) := x(t_j; x_0^{(n)}).
\end{equation*}

\subsection{Propagation along characteristic trajectories}
We parameterize the velocity field by a neural network $v_\theta(t,x)$. The particle trajectories are generated by the neural ODE
\begin{equation*}
\dot{x}_t = v_\theta(t,x_t),
\qquad
x_0 \sim \rho_0/M.
\end{equation*}

The auxiliary variable
\[
q_t=\nabla_x H'(\rho_t)
\]
is propagated along the same particle trajectories. Together with the
density evolution along the trajectories, the coupled dynamics are given by
\begin{equation*}
\begin{aligned}
\dot{\rho}_t
&=
-\rho_t\,\nabla_x\cdot v_\theta(t,x_t),\\
\dot{q}_t
&=
-(\nabla_x v_\theta(t,x_t))^\top q_t
-\rho_t H''(\rho_t)\,
\nabla_x(\nabla_x\cdot v_\theta(t,x_t))\\
&\quad
-(\nabla_x\cdot v_\theta(t,x_t))
\left(
1+\frac{\rho_t H'''(\rho_t)}{H''(\rho_t)}
\right)q_t.
\end{aligned}
\end{equation*}

Using the Euler time-discretization, the coupled dynamics become
\begin{equation*}
\begin{aligned}
x_{t_{j+1}}^{(n)}
&=
x_{t_j}^{(n)}
+
\Delta t\,v_\theta(t_j,x_{t_j}^{(n)}),\\
\rho_{t_{j+1}}^{(n)}
&=
\rho_{t_j}^{(n)}
-
\Delta t\,\rho_{t_j}^{(n)}
\nabla_x\cdot v_\theta(t_j,x_{t_j}^{(n)}),\\
q_{t_{j+1}}^{(n)}
&=
q_{t_j}^{(n)}
-
\Delta t
\Bigg[
(\nabla_x v_\theta(t_j,x_{t_j}^{(n)}))^\top q_{t_j}^{(n)} +\rho_{t_j}^{(n)}H''(\rho_{t_j}^{(n)})
\nabla_x(\nabla_x\cdot v_\theta(t_j,x_{t_j}^{(n)}))
\\
&\qquad\qquad
+(\nabla_x\cdot v_\theta(t_j,x_{t_j}^{(n)}))
\left(
1+
\frac{
\rho_{t_j}^{(n)}H'''(\rho_{t_j}^{(n)})
}{
H''(\rho_{t_j}^{(n)})
}
\right)
q_{t_j}^{(n)}
\Bigg].
\end{aligned}
\end{equation*}

At $t_0=0$, the density and generalized score are initialized by
\[
\rho_0^{(n)}=\rho_0(x_0^{(n)}),
\qquad
q_0^{(n)}
=
\nabla_x H'\!\left(\rho_0(x)\right)|_{x=x_0^{(n)} },
\]
where $x_0^{(n)}$ are sampled from the initial density $\rho_0$.
Thus, both the particle positions $x_{t_j}^{(n)}$ and the $q^{(n)}_{t_j}$ can be computed recursively along the trajectories.
One important distinction is that 
\[
\mu_0^{N_p}:= \frac{M}{N_p}\sum_{n=1}^{N_p} \delta_{x_0^{(n)}}
\]
is the empirical approximation of the probability measure $
\rho_0(x)\,dx$, 
while $
\rho_0^{(n)}=\rho_0\bigl(x_0^{(n)}\bigr)$ is a quantity carried along each characteristic.


\begin{remark}[Asymptotic mass preservation] Forward Euler satisfies the Lagrangian mass relation up to an $O(\Delta t^2)$ local error (asymptotically exact as $\Delta t \to 0$), a property shared by similar characteristic-based neural methods
\cite{Eric23,zhou2025}. Alternatively, exact discrete mass preservation can be enforced directly via:
\[
\rho_{t_{j+1}}^{(n)}
=
\frac{\rho_{t_j}^{(n)}}
{\det\!\left(
I+\Delta t\,\nabla_x v_\theta
(t_j,x_{t_j}^{(n)})
\right)}.
\]
Indeed, the forward Euler characteristic map satisfies
\[
J_{j+1}^{(n)}
=
J_j^{(n)}
\det\!\left(
I+\Delta t\,\nabla_x v_\theta
\bigl(t_j,x_{t_j}^{(n)}\bigr)
\right),
\]
and hence
\[
\rho_{j+1}^{(n)}J_{j+1}^{(n)}
=
\rho_j^{(n)}J_j^{(n)}.
\]
Thus, the discrete Lagrangian mass relation is preserved exactly.

Furthermore, in the limit  $N_p\to\infty$, the particle approximation recovers the corresponding continuum (in space) evolution. For $\Delta t $ sufficiently small, $\rho_{t_{j+1}}= (I+\Delta t v_\theta(t_j,x))_\#\rho_{t_{j}}$, we obtain   
$$
E(\rho_{t_{j+1}})=
E(\rho_{t_j})+\Delta t
\int_\Omega
\rho_{t_j}
\nabla_x\mu[\rho_{t_j}]
\cdot v_\theta(t_j,x)\,dx
+O(\Delta t^2).
$$
Hence 
$$
E(\rho_{t_{j+1}})-
E(\rho_{t_j})=-\Delta t \int_\Omega \rho_{t_j} |\nabla\mu[\rho_{t_j}]|^2dx
+ \Delta t \int_\Omega \rho_{t_j} \nabla\mu[\rho_{t_j}]\cdot e_\theta dx + O(\Delta t^2), 
$$
where the local error $e_\theta =v_\theta(t_j, x) +\nabla_x \mu[\rho_{t_j}(x)]$.  Thus, energy dissipation persists provided that the training error remains sufficiently small relative to the intrinsic energy dissipation.
\end{remark}

\begin{remark}[Conditional positivity preservation] Forward Euler preserves density positivity under the time-step constraint: 
$$
\Delta t < \frac{1}{\|\nabla_x\cdot v_\theta\|_{L^\infty([0,T]\times\Omega)}}.
$$
Alternatively, positivity is preserved unconditionally by using the density update  
$$
\rho_{t_{j+1}}^{(n)}= \rho_{t_{j}}^{(n)}\exp\left(-\Delta t \nabla_x\cdot v_\theta(t_j, x^{(n)}_{t_j})\right). 
$$
\end{remark}

\subsection{Interaction potential}

For each fixed time $t_j$, we approximate the density $\rho(t_j, \cdot)$ by the empirical particle measure 
$$
\mu_{t_j}^{N_p} := \frac{M}{N_p}
\sum_{m=1}^{N_p}
\delta_{x_{t_j}^{(m)}},
$$
where
$
M: =\int_\Omega\rho_0(x)dx
$
is the total mass. 
Consequently,
\begin{equation}\label{W_app}
\nabla(W\ast\rho)(x_{t_j}^{(n)})
\approx
\frac{M}{N_p}
\sum_{m=1}^{N_p}
\nabla W(x_{t_j}^{(n)}-x_{t_j}^{(m)}) := F_{t_j}^{(n)}.
\end{equation}
Hence, the interaction force is evaluated directly over the particle positions. 

\begin{remark} 
The direct pairwise evaluation of the nonlocal interaction term incurs a 
computational complexity of $\mathcal{O}(N_p^2)$ per time step, where $N_p$ 
is the total number of particles. For the particle counts used in this work, this cost remained manageable. To overcome this quadratic bottleneck in larger simulations, fast summation methods \cite{GR1987, GS1991} can reduce complexity to $\mathcal{O}(N_p \log N_p)$ or even $\mathcal{O}(N_p)$.
\end{remark} 

\begin{remark}[Singular interaction potentials]
For singular interaction potentials, the self-interaction term ($m=n$) is set to zero in the numerical implementation, so the singular value $\nabla W(0)$ will not be evaluated.
\end{remark}

\subsection{Numerical algorithm}
Combining the particle approximation with the propagated quantity of $q$ yields the following discrete objective functional: 
\begin{equation}\label{L_D_def}
L_D(\theta)
=
\frac{M}{N_p}
\sum_{n=1}^{N_p}
\sum_{j=0}^{N_t-1}
\left|
v_\theta(t_j,x_{t_j}^{(n)})
+
q_{t_j}^{(n)}
+
\nabla V(x_{t_j}^{(n)})
+
F_{t_j}^{(n)}
\right|^2
\Delta t.
\end{equation}
Here, the first term is the neural approximation of the velocity field, the second is the auxiliary variable propagated along the particle trajectories, and the last two terms represent the external and nonlocal interaction forces, respectively.

We minimize $L_D(\theta)$ with respect to the neural-network parameters $\theta$ using an appropriate optimizer. At each optimization step, the particle trajectories and $q$ are updated recursively, while the interaction force is evaluated by \eqref{W_app}. The resulting procedure provides a fully particle-based approximation of \eqref{opt0}.

The complete numerical procedure is summarized in Algorithm \ref{alg:interaction-flow}.

\begin{remark}[Computational cost]
Spatial derivatives of the neural velocity
field, up to second-order derivatives, are evaluated via automatic differentiation.
While effective in our implementation, their evaluations introduce  computational and
memory overhead that scales with spatial dimension and network size.
\end{remark}

\begin{algorithm}[htbp]
\caption{Generalized score-based neural transport 
}
\label{alg:interaction-flow}
\begin{algorithmic}[1]
\Require Time horizon $T$, time steps $N_t$ with $\Delta t=T/N_t$, particle number $N_p$, initial density $\rho_0$, learning rate $\eta$, optimization steps $N_{\mathrm{step}}$
\State Initialize parameters $\theta$ of $v_\theta$
\State Set time grid $\{t_j\}_{j=0}^{N_t}$
\For{$\mathrm{step}=1,2,\cdots,N_{\mathrm{step}}$}
\State Sample $\{x_0^{(n)}\}_{n=1}^{N_p}\sim \rho_0/M$, set $\displaystyle x_{t_0}^{(n)}=x_0^{(n)}$
\State Initialize $\displaystyle \rho_{t_0}^{(n)}=\rho_0(x_0^{(n)})$
\State Initialize $\displaystyle q_{t_0}^{(n)}=\nabla_x H'(\rho_0(x_0^{(n)}))$
\State Set $\mathrm{Loss}=0$
\For{$j=0,1,\cdots,N_t-1$}
\State Compute interaction force
\State $\displaystyle F_{t_j}^{(n)}=\frac{M}{N_p}\sum_{m=1}^{N_p}\nabla W(x_{t_j}^{(n)}-x_{t_j}^{(m)})$
\State Accumulate objective
\State $\displaystyle \mathrm{Loss}=\mathrm{Loss}+\frac{M}{N_p}\sum_{n=1}^{N_p}\left|v_\theta(t_j,x_{t_j}^{(n)})+q_{t_j}^{(n)}+\nabla V(x_{t_j}^{(n)})+F_{t_j}^{(n)}\right|^2\Delta t$
\State Update particle position
\State $\displaystyle x_{t_{j+1}}^{(n)}=x_{t_j}^{(n)}+\Delta t\,v_\theta(t_j,x_{t_j}^{(n)})$
\State Update density variable
\State $\displaystyle \rho_{t_{j+1}}^{(n)}
=\rho_{t_j}^{(n)}
-\Delta t\,\rho_{t_j}^{(n)}
\nabla_x\cdot v_\theta(t_j,x_{t_j}^{(n)})$
\State Update auxiliary variable
\State $\displaystyle
\begin{aligned}
q_{t_{j+1}}^{(n)}
={}&
q_{t_j}^{(n)}
-\Delta t
\Bigg[
(\nabla_x v_\theta(t_j,x_{t_j}^{(n)}))^\top q_{t_j}^{(n)}
+\rho_{t_j}^{(n)}H''(\rho_{t_j}^{(n)})
\nabla_x(\nabla_x\cdot v_\theta(t_j,x_{t_j}^{(n)}))
\\
&
+(\nabla_x\cdot v_\theta(t_j,x_{t_j}^{(n)}))
\left(
1+\frac{\rho_{t_j}^{(n)}H'''(\rho_{t_j}^{(n)})}
{H''(\rho_{t_j}^{(n)})}
\right)
q_{t_j}^{(n)}
\Bigg].
\end{aligned}$
\EndFor
\State Compute $\nabla_\theta \mathrm{Loss}$ by Autodiff through the trajectory $\{x_{t_j}^{(n)},\rho_{t_j}^{(n)},q_{t_j}^{(n)}\}_{j=0}^{N_t}$
\State Update $\theta$ with Adam using $\nabla_\theta \mathrm{Loss}$
\EndFor
\end{algorithmic}
\end{algorithm}

The computational graph is retained through the entire trajectory, so the dependence of both $x_{t_j}^{(n)}$ and $q_{t_j}^{(n)}$ on $\theta$ is included when computing $\nabla_\theta L_D$
by automatic differentiation.

\section{Error Analysis}
In this section, we analyze the error associated with Algorithm
\ref{alg:interaction-flow}. We restrict our analysis to the entropy case
\begin{equation*}
    H(\rho)=\rho\log\rho-\rho,
\end{equation*}
for which
\begin{equation*}
    q=\nabla_x H'(\rho)=\nabla_x\log\rho=s.
\end{equation*}
For a given parameter $\theta$, let $\rho_\theta$ denote the density generated  by the neural velocity field $v_\theta$. Specifically, 
\begin{equation*}
    \rho_\theta(t,\cdot)=(x_t)_\#\rho_0,
    \qquad
    \dot{x}_t=v_\theta(t,x_t), \quad x_0\sim\rho_0/M.
\end{equation*}
Since the initial particles $\{x_0^{(n)}\}_{n=1}^{N_p}$ are i.i.d. samples from $\rho_0/M$, their corresponding exact particle trajectories remain i.i.d. for each fixed $t$, with
\begin{equation*}
    x_t^{(n)}\sim\rho_\theta(t,\cdot)/M, \qquad n=1,\ldots,N_p.
\end{equation*}
We define
\begin{equation*}
\theta_S
=
\arg\min_{\theta} L_S(\theta)
\end{equation*}
to be a minimizer of the continuous,  population-level objective, 
\begin{equation*}
L_S(\theta)
=
\int_0^T\int_\Omega
\left|v_\theta+\nabla\mu_\theta\right|^2
\rho_\theta\,dxdt,
\end{equation*}
where $\mu_\theta := \mu[\rho_{\theta}] = H'(\rho_\theta) + V + W \ast \rho_\theta$ denotes the corresponding chemical potential. Similarly, we define
\begin{equation*}
\theta_D
=
\arg\min_{\theta} L_D(\theta),
\end{equation*}
where $L_D(\theta)$ is defined in \eqref{L_D_def}.

Let $\theta_k$ denote the parameter obtained after $k$ iterations. Since $\theta_D$ minimizes the discrete objective $L_D$, we have $L_D(\theta_S)\geq L_D(\theta_D)$. Therefore, the population risk at $\theta_k$ can be bounded as  
\begin{equation}\label{L_Decomp_1}
\begin{aligned}
L_S(\theta_k) & \leq L_S(\theta_k) + L_D(\theta_S) - L_D(\theta_D) := \varepsilon_{\mathrm{app}} + \varepsilon_{\mathrm{opt}} + \varepsilon_{\mathrm{MD}}.
\end{aligned}
\end{equation}
Here $$\varepsilon_{\mathrm{app}} := L_S(\theta_S)$$ denotes the approximation error,
$$\varepsilon_{\mathrm{opt}} := L_D(\theta_k)-L_D(\theta_D)$$ represents the optimization error with respect to the discrete objective,
while $$\varepsilon_{\mathrm{MD}}:= \Bigl[
L_S(\theta_k)-L_D(\theta_k) \Bigr] +\Bigl[L_D(\theta_S)-L_S(\theta_S)\Bigr]$$ accounts for the discrepancy between the population and discrete objectives.
Here $L_D(\theta_k)$ is understood to be evaluated on a batch $N_p$ particles, which is independent of the training data. 

Then we first state the main error estimate.

\begin{assumption}\label{Ass_thm}
Let $\rho^*$ denote the strictly positive, smooth solution to \eqref{Ori_sys}, and set
$s^*:=\nabla_x\log\rho^*.$
Assume that:
\begin{itemize}
    \item[(i)] the potentials are smooth so that
    \begin{equation*}
    \nabla_x V\in W^{1,\infty}(\Omega),
    \qquad
    \nabla_x W\in L^1(\Omega)\cap W^{1,\infty}(\Omega);
    \end{equation*}
    \item[(ii)] for the parameters $\theta$ considered in the error analysis,
    \begin{equation*}
    v_\theta,\ \nabla_xv_\theta,\ \nabla_x(\nabla_x\cdot v_\theta)
    \in L^\infty([0,T]\times\Omega),
    \end{equation*}
    and $v_\theta$ is Lipschitz continuous in $t$, uniformly in $x\in\Omega$, i.e.,
    $\partial_t v_\theta\in L^\infty([0,T]\times\Omega)$.
\end{itemize}
\end{assumption}

The main result is summarized in Theorem \ref{Thm_total_error}.

\begin{theorem}\label{Thm_total_error}
Assume that the initial density $\rho_0$ is strictly positive and sufficiently smooth, with total mass $M:=\int_\Omega \rho_0\,dx$.
Let $\rho^*$ be the smooth solution for \eqref{Ori_sys}. Under Assumption~\ref{Ass_thm}, let $\theta_k$ denote the parameter obtained after $k$ iterations.

Then, for any $\delta\in(0,1)$, with probability at least $1-\delta$,
\begin{equation*}
L_S(\theta_k)
\le
C_S
\int_0^T
\|v_{\theta_S}(t,\cdot)-v^*(t,\cdot)\|^2_{W^{2,\infty}(\Omega)}
\,dt
+
\varepsilon_{\mathrm{opt}}
+
C_{\mathrm{MD}}
\left(
\sqrt{\frac{\log(N_p N_t/\delta)}{N_p}}
\;+\;
\Delta t
\right),
\end{equation*}
where
\begin{equation*}
v^*
=
-\nabla_x\log\rho^*
-\nabla_xV
-\nabla_x(W\ast\rho^*).
\end{equation*}
Here $C_S$ and $C_{\mathrm{MD}}$
are both independent of $\theta, N_p$ and $\Delta t$.
\end{theorem}
\begin{remark}[Scope of Theorem~5.2 and Treatment of Singular Kernels]
\label{rem:singular_kernel}
The theoretical convergence bound in Theorem~5.2 assumes $\nabla W \in L^1(\Omega) \cap W^{1,\infty}(\Omega)$, which implies that the interaction force is globally bounded and Lipschitz continuous. 
The 2D Patlak--Keller--Segel model utilizes the Newtonian potential interaction $\nabla W(x) = \frac{\chi}{2\pi} \frac{x}{|x|^2}$, which possesses a $O(|x|^{-1})$ singularity at the origin and does not belong to $W^{1,\infty}(\Omega)$. To bridge this gap:
\begin{enumerate}
    \item \textbf{Theoretical Coverage:} The theoretical estimates apply to a mollified/regularized kernel $\nabla W_\varepsilon := \nabla W \ast \eta_\varepsilon \in W^{1,\infty}(\Omega)$, where $\eta_\varepsilon$ is a smooth mollifier with regularization parameter $\varepsilon > 0$.
    \item \textbf{Numerical Implementation:} In our numerical experiments for Keller--Segel, we employ the blob-regularized kernel:
    \[
    \nabla W_\varepsilon(x) = \frac{\chi}{2\pi} \frac{x}{|x|^2 + \varepsilon^2}, \quad \text{with } \varepsilon > 0,
    \]
    which satisfies $\nabla W_\varepsilon \in W^{1,\infty}(\Omega)$ for any fixed $\varepsilon > 0$. Self-interaction ($x=0$) naturally vanishes under this formulation.
\end{enumerate}
\end{remark}

\begin{remark}[Interpretation of Theorem \ref{Thm_total_error}]
Theorem~ \ref{Thm_total_error} provides a conditional error decomposition rather than an unconditioned convergence guarantee. It isolates the contributions of neural network approximation capacity, optimization residuals, and sampling error. To convert this error bound into an asymptotic convergence result, one must complement Theorem \ref{Thm_total_error} with explicit hypotheses on neural network expressivity and optimization performance.
\end{remark}
\begin{corollary}[End-to-End Asymptotic Convergence]
\label{cor:convergence}
Let the hypotheses of Theorem \ref{Thm_total_error} hold. Assume in addition:
\begin{enumerate}
    \item \textbf{(Approximation Property):} The true velocity field $v^*$ satisfies $v^* \in C([0,T]; W^{2,\infty}(\Omega;\mathbb{R}^d))$. For every $\epsilon > 0$, there exists a sequence of neural network architectures $\mathcal{N}_m$ with parameters $\theta \in \mathbb{R}^{m}$ such that:
    \[
    \inf_{\theta \in \mathbb{R}^{m}} \int_0^T \| v_{\theta}(t, \cdot) - v^*(t, \cdot) \|_{W^{2,\infty}(\Omega)}^2 \, \mathrm{d}t \le \epsilon \quad \text{as } m \to \infty.
    \]
    \item \textbf{(Optimization Guarantee):} For each architecture $\mathcal{N}_m$, the training optimization algorithm (e.g., gradient descent) achieves a parameter $\theta_k$ at iteration $k$ satisfying an optimization error $\varepsilon_{\text{opt}}(k, m)$ such that $\lim_{k, m \to \infty} \varepsilon_{\text{opt}}(k, m) = 0$.
\end{enumerate}
Then, taking the joint limit of network capacity $m \to \infty$, optimization iterations $k \to \infty$, sample sizes $N_p, N_t \to \infty$, and step size $\Delta t \to 0$, the learned velocity field $v_{\theta_k}$ converges to $v^*$ in probability, and the total empirical loss satisfies:
\[
L_S(\theta_k) \xrightarrow{\quad P \quad} 0.
\]
\end{corollary}

\begin{remark}[Discussion on Approximation and PL Conditions]
Assumption~(1) is supported by classical universal approximation results
and recent quantitative neural-network approximation theory in Sobolev
norms, including results for smooth activation functions such as $\tanh$
\cite{DeRyck2021,Girault2024, Hornik1991}, which provide approximation
guarantees for both functions and their derivatives. Assumption~(2) is
an optimization condition related to the Polyak--{\L}ojasiewicz (PL)
property. PL-type convergence guarantees have been established for
overparameterized neural networks under suitable conditions, with recent
advances further refining the corresponding optimization regimes
\cite{Banerjee2023,Liu2022}. We state Corollary~\ref{cor:convergence}
conditionally on these approximation and optimization assumptions,
rather than asserting that they hold for the neural-network loss
considered here in full generality.
\end{remark}

Next, we show how to establish the bounds for the three error terms appearing in \eqref{L_Decomp_1}. 
We first estimate the approximation error $\varepsilon_{\mathrm{app}}$.

\begin{lemma}\label{Lemma_LD}
Under the assumptions of Theorem \ref{Thm_total_error}, we have
\begin{equation*}
\begin{aligned}
\varepsilon_{\mathrm{app}}
\le\;&
C_S
\int_0^T
\|v_{\theta_S}(t,\cdot)-v^*(t,\cdot)\|^2_{W^{2,\infty}(\Omega)}
\,dt,
\end{aligned}
\end{equation*}
where $v^* = - \nabla_x \log \rho^* - \nabla_x V - \nabla_x( W \ast \rho^*)$, and $C_S > 0$ is independent of $N_p$ and $\Delta t$. 

\end{lemma}

Finally, we estimate the Monte-Carlo and discretization error $\varepsilon_{\mathrm{MD}}$.
The quantity $\varepsilon_{\mathrm{MD}}$ involves the empirical average over the $N_p$
randomly sampled particles $\{x_0^{(n)}\}_{n=1}^{N_p}$, as well as the empirical
approximation of the interaction force in \eqref{W_app}, and is therefore itself a random
variable. We control $\varepsilon_{\mathrm{MD}}$ with high probability, via a concentration
inequality; we denote by $\delta \in (0,1)$ the associated failure probability.

\begin{lemma}\label{Lemma_MC}
Under the assumptions of Theorem \ref{Thm_total_error}, for any $\delta \in (0,1)$, with
probability at least $1-\delta$,
\begin{equation*}
\varepsilon_{\mathrm{MD}}
\le
C_{\mathrm{MD}}
\left(
\sqrt{\frac{\log(N_p N_t/\delta)}{N_p}}
\;+\;
\Delta t
\right),
\end{equation*}
where $N_t = T/\Delta t$, and $C_{\mathrm{MD}}>0$ is independent of $N_p$ and $\Delta t$.
\end{lemma}

We are now ready to prove Theorem \ref{Thm_total_error}.

\begin{proof}[Proof of Theorem \ref{Thm_total_error}]
By the error decomposition \eqref{L_Decomp_1}, 
\begin{equation*}
L_S(\theta_k)
\le
\varepsilon_{\mathrm{app}}
+
\varepsilon_{\mathrm{opt}}
+
\varepsilon_{\mathrm{MD}}.
\end{equation*}
The approximation error is bounded in Lemma \ref{Lemma_LD} as
\begin{equation*}
\varepsilon_{\mathrm{app}}
\le
C_S
\int_0^T
\|v_{\theta_S}(t,\cdot)-v^*(t,\cdot)\|^2_{W^{2,\infty}(\Omega)}
\,dt.
\end{equation*}
Combining this estimate with Lemma \ref{Lemma_MC}, we obtain, with probability at least
$1-\delta$,
\begin{equation*}
L_S(\theta_k)
\le
C_S
\int_0^T
\|v_{\theta_S}(t,\cdot)-v^*(t,\cdot)\|^2_{W^{2,\infty}(\Omega)}
\,dt
+
\varepsilon_{\mathrm{opt}}
+
C_{\mathrm{MD}}
\left(
\sqrt{\frac{\log(N_p N_t/\delta)}{N_p}}
\;+\;
\Delta t
\right).
\end{equation*}
This establishes the desired error bound and completes the proof.
\end{proof}

\section{Numerical experiments and results}
In this section, we present a series of numerical experiments to evaluate the proposed neural flow-matching method. We first examine its numerical convergence and free-energy dissipation behavior, and then consider several nonlinear and nonlocal aggregation-- diffusion models to demonstrate the applicability of the proposed framework. Finally, we investigate its performance in higher-dimensional settings.

\begin{example}\label{DOI}
We first consider the one-dimensional Doi--Onsager equation with the Maier--Saupe interaction potential
\begin{equation*}
\partial_t \rho
=
\partial_x
\left(
\partial_x \rho
+
\alpha \rho \partial_x(W * \rho)
\right),
\qquad
t>0,
\quad
x\in[0,2\pi],
\end{equation*}
where the interaction kernel is given by
\begin{equation*}
    W(x)=\sin^2(x).
\end{equation*}

The initial condition is chosen as
\begin{equation*}
    \rho(x,0)
    =
    \frac{x+1}{2\pi(\pi+1)},
\end{equation*}
subject to zero-flux boundary conditions. The interaction strength is taken as
$\alpha = 5,$ and the time step size is $\Delta t = 0.01.$ The neural velocity field is parameterized by
\begin{equation*}
    v_\theta(t,x)
    =
    x(2\pi-x)\mathcal N_\theta(t,x),
\end{equation*}
which approximately enforces the no-flux structure near the boundary.

We employ the multi-stage training strategy described in Appendix~\ref{AppendixB}.
Figure~\ref{fig:multistage_doi_onsager} shows the multi-stage training loss
together with the evolution of the recovered density at different time instances.

Figure~\ref{fig:doi_dt} examines the effects of temporal discretization and optimization error.
With $5000$ optimizaiton iterations, the $L^1$ and $L^2$ errors converge approximately at 
first-order in $\Delta t$, while the KL error decays at approximately second-order. For smaller optimization budgets, the errors tend to level off as $\Delta t$ decreases, indicating dominance of the optimization error. 

Figure~\ref{fig:doi_Np} shows the dependence on the  particle numbers $N_p$. All three errors decrease with $N_p$, with faster improvement for larger optimization budgets. For $5000$ iterations, the observed rates are consistent
with $N_p^{-1/2}$ for the $L^1$ and $L^2$ errors and $N_p^{-1}$ for the KL error. The error bars show robustness with respect to random initialization.
\end{example}

\begin{figure}[htbp]
    \centering
    \begin{minipage}{0.45\textwidth}
        \centering
        \includegraphics[width=\linewidth]{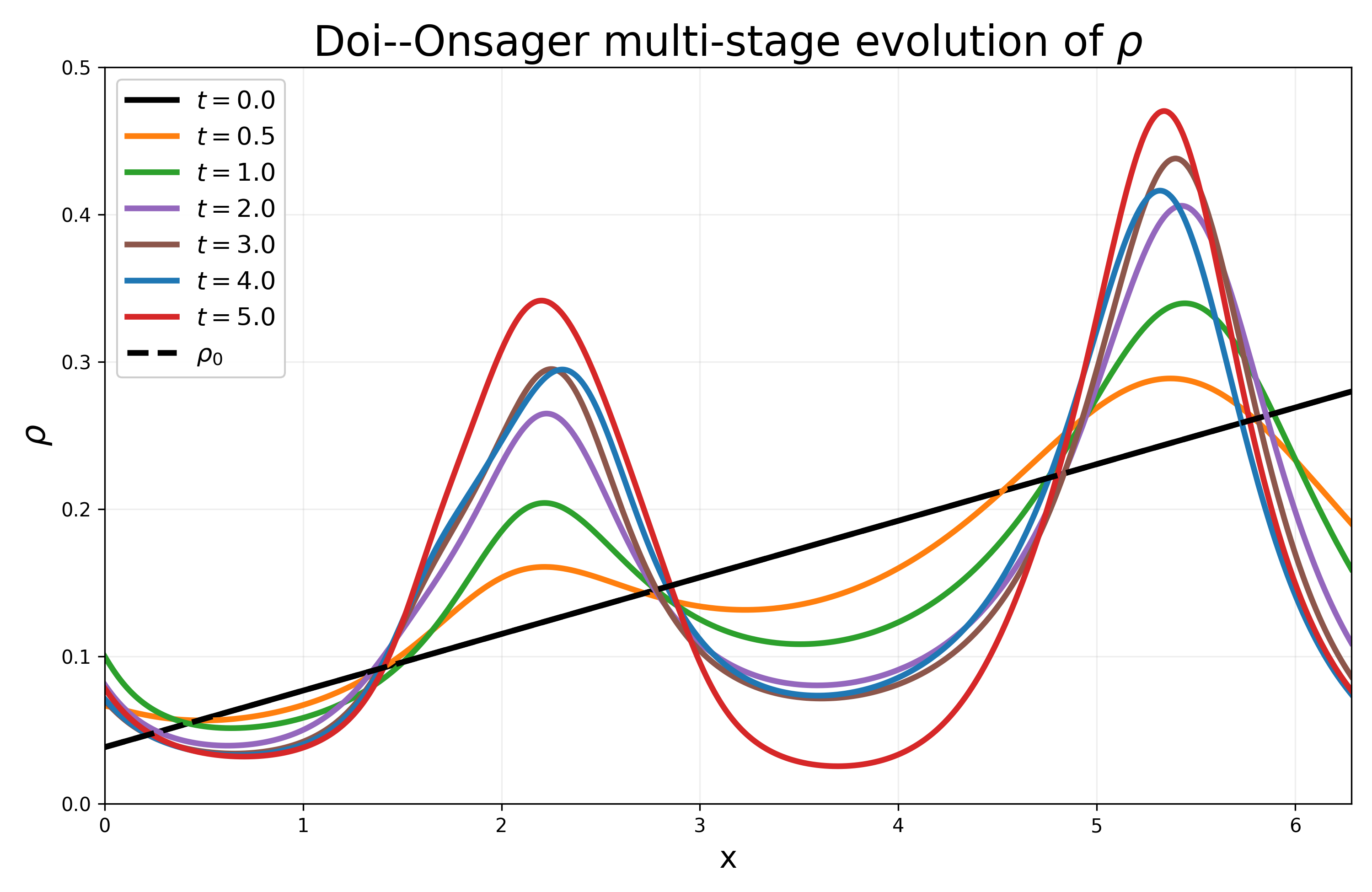}
    \end{minipage}
    \hfill
    \begin{minipage}{0.45\textwidth}
        \centering
        \includegraphics[width=\linewidth]{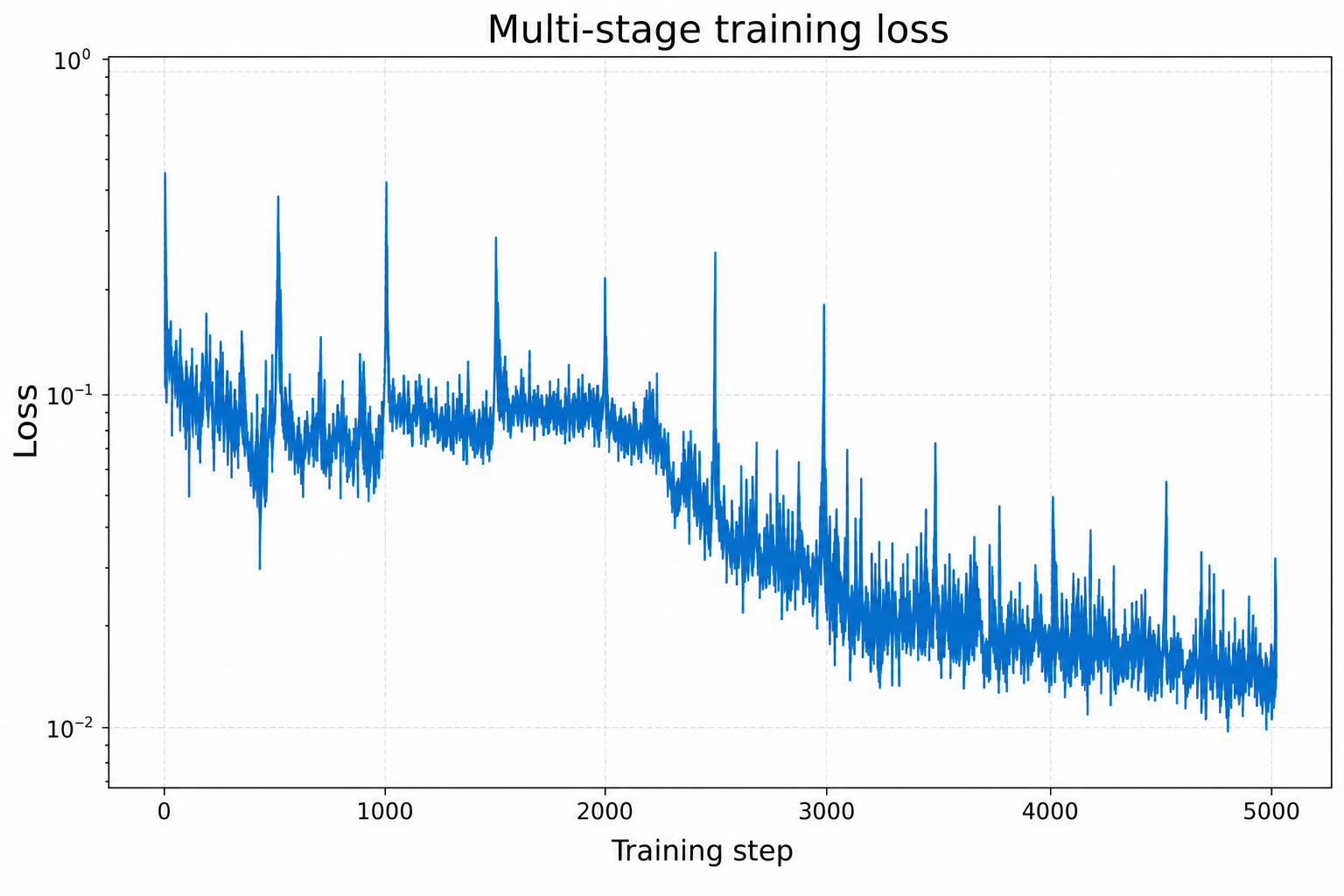}
    \end{minipage}

    \caption{
    Evolution of $\rho$ and multi-stage training loss for
    Example~\ref{DOI}.
    }
    \label{fig:multistage_doi_onsager}
\end{figure}

\begin{figure}[htbp]
    \centering
    \includegraphics[width=0.82\textwidth]{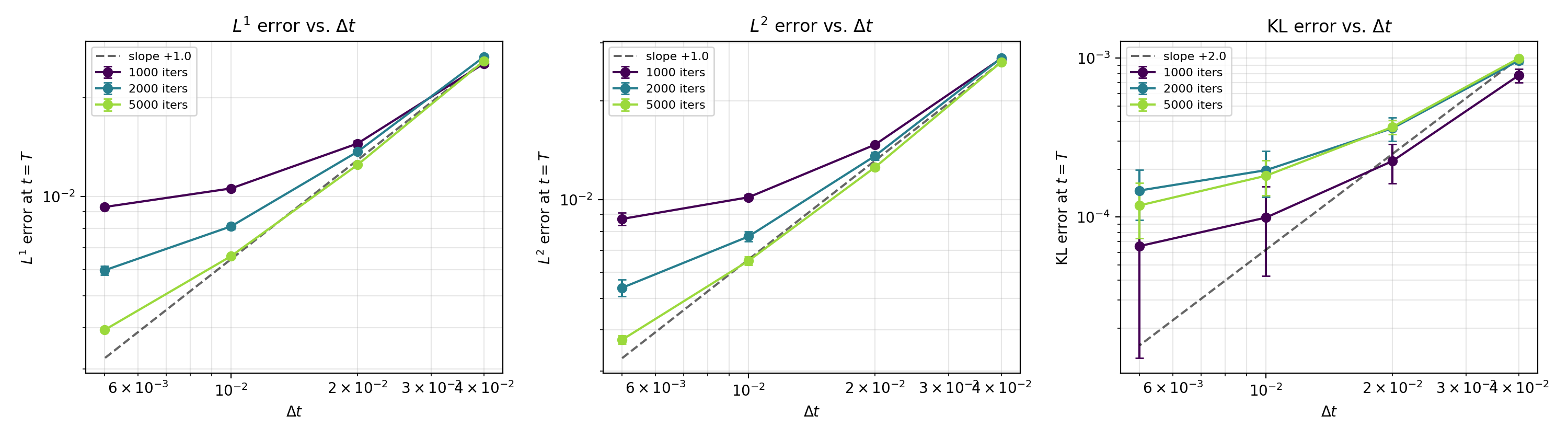}
    \caption{
    Errors with respect to $\Delta t$ for Example~\ref{DOI}.
    }
    \label{fig:doi_dt}
\end{figure}

\begin{figure}[htbp]
    \centering
    \includegraphics[width=0.82\textwidth]{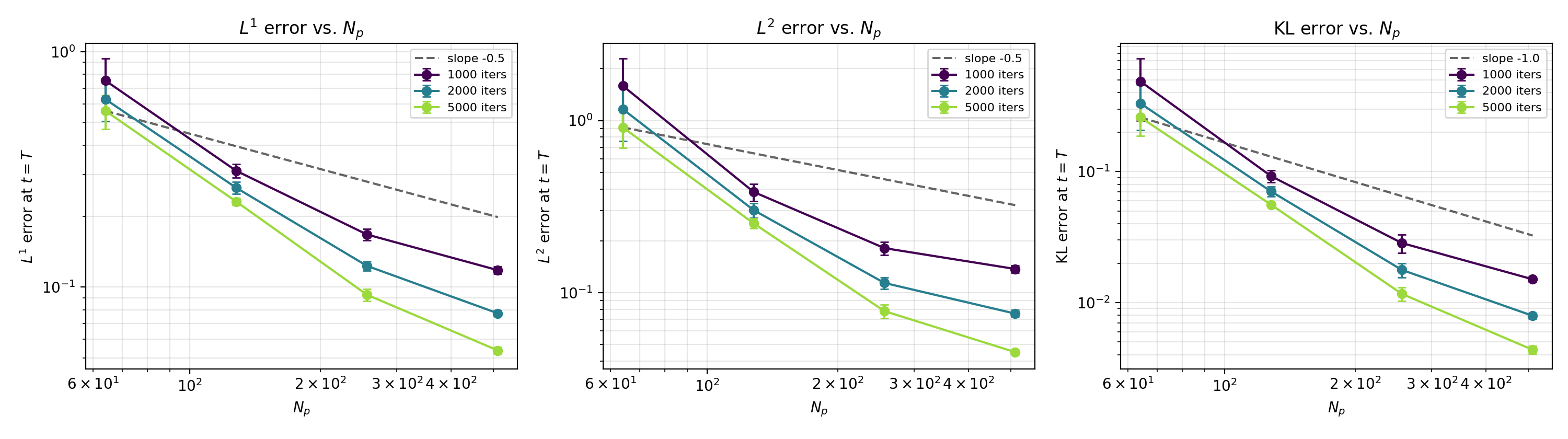}
    \caption{
    Errors with respect to $N_p$ for Example~\ref{DOI}.
    }
    \label{fig:doi_Np}
\end{figure}

\begin{example}\label{PKS_2d_sub}
(Patlak--Keller--Segel system)
In this example, we test the proposed framework on the two-dimensional Patlak--Keller--Segel system, which models cell aggregation driven by chemotaxis and has been considered in various numerical studies \cite{CLASSD10liu_positive_2020,CT26, MA14}. Consider
\begin{equation*}
\partial_t \rho
=
\nabla_x \cdot
\left(
\nabla_x \rho
+
\rho \nabla_x (W * \rho)
\right),
\qquad
t>0,
\quad
x \in \Omega,
\end{equation*}
where
\begin{equation*}
W(x)
=
\frac{\chi}{2\pi}
\log |x|,
\qquad
x \in \mathbb{R}^2.
\end{equation*}

The interaction force is therefore given by
\begin{equation*}
\nabla_x W(x)
=
\frac{\chi}{2\pi}
\frac{x}{|x|^2}.
\end{equation*}

We consider the bounded domain $ \Omega = [-5,5]\times[-5,5]$, subject to zero-flux boundary conditions. In the experiments, we set $\chi=1$ for the sub-critical case and $\chi=30$ for the super-critical case. The initial density is chosen as a smooth approximation of the characteristic function on
\(
[-1,1]\times[-1,1]
\),
namely
\begin{equation*}
\rho_0(x,y)
=
\varepsilon
+
C_0
\sigma(k(x+1))
\sigma(k(1-x))
\sigma(k(y+1))
\sigma(k(1-y)),
\end{equation*}
where
\(
\sigma(z)=1/(1+e^{-z})
\)
denotes the sigmoid function, \(k>0\) controls the sharpness of the transition layer, and the normalization constant \(C_0\) is chosen such that $\int_\Omega \rho_0(x,y)\,dxdy = 1$.

We employ the multi-stage neural transport algorithm together with the score ODE evolution to propagate the density and recover the solution profile at different time instances. The resulting density evolutions in the sub-critical and super-critical regimes are shown in Figures \ref{fig:pks_2d_evolution1} and \ref{fig:pks_2d_evolution2}, respectively.

{For the super-critical case ($t\in [0, 0.5]$), 
the computation exhibits increasingly strong concentration prior to blow up, rather than post-blow-up dynamics. 
Near-field singular forces are  
evaluated using the blob-regularized kernel (Remark~5.3), with self-interaction vanishing at $x=0$.}
\end{example}

\begin{figure}[htbp]
    \centering
    \includegraphics[width=0.8\textwidth]{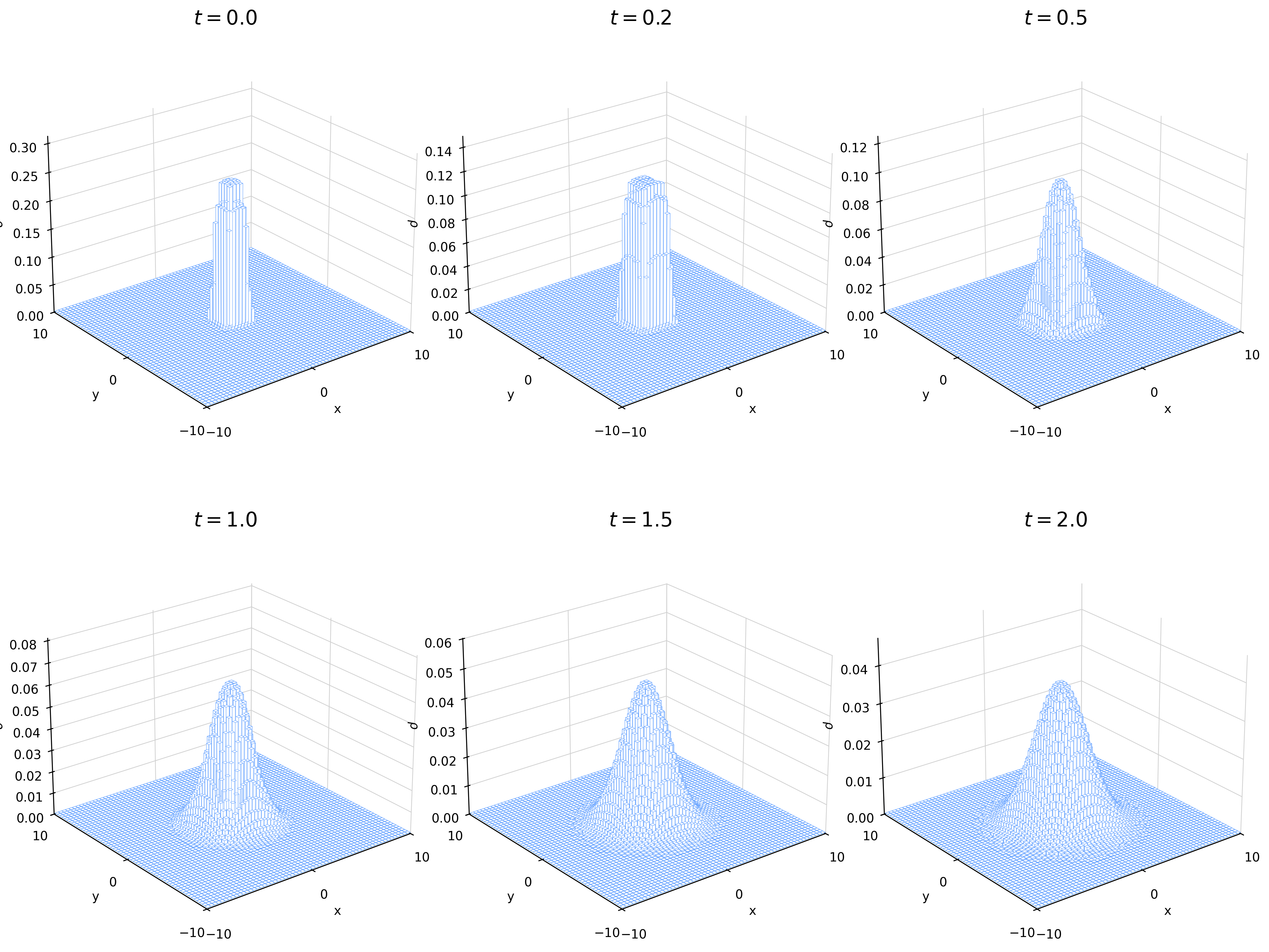}
    \caption{
    Evolution of  \(\rho\) for two-dimensional Patlak--Keller--Segel system for Example \ref{PKS_2d_sub} (sub-critical).}
    \label{fig:pks_2d_evolution1}
\end{figure}

\begin{figure}[htbp]
    \centering
    \includegraphics[width=0.8\textwidth]{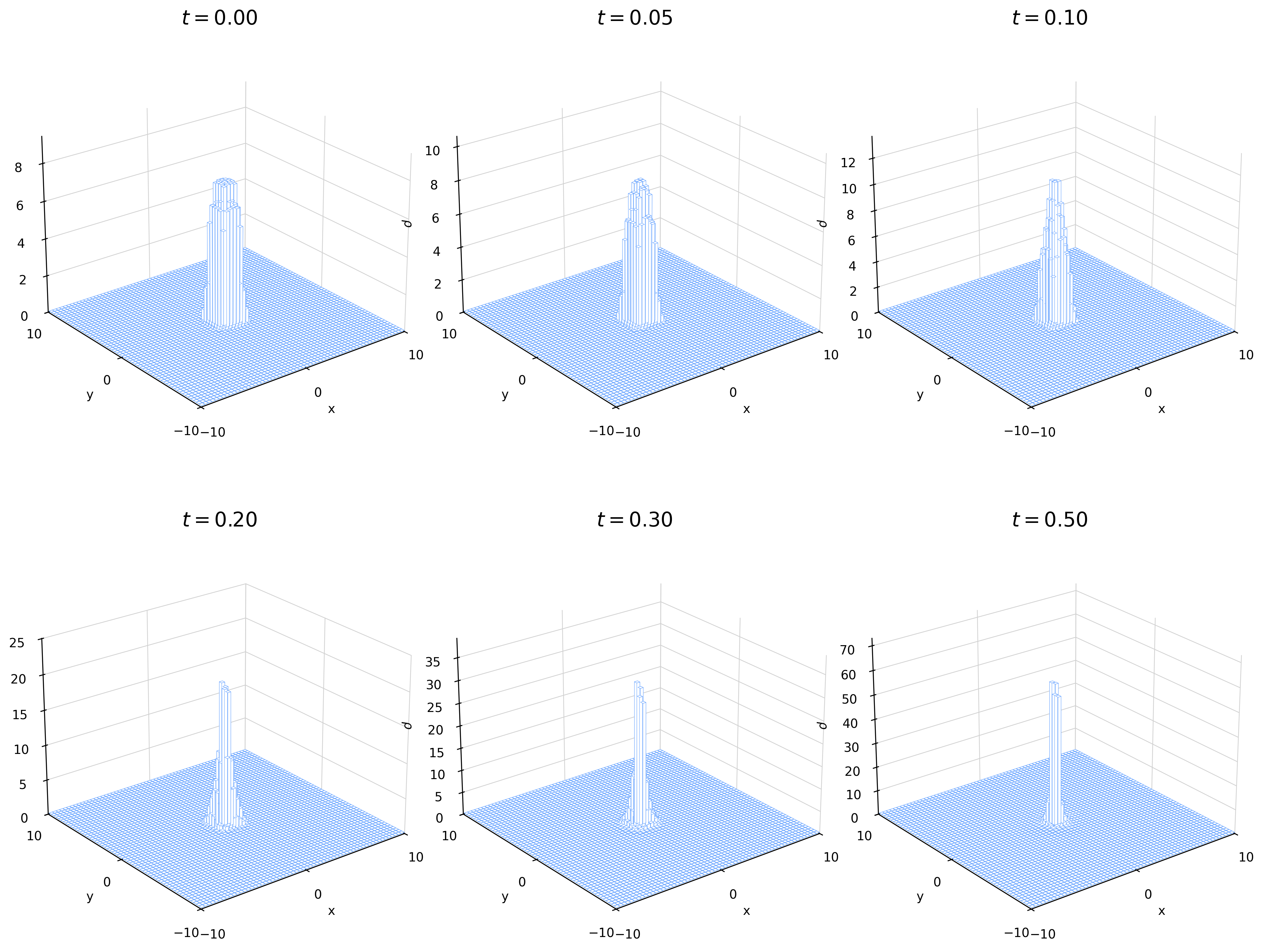}
    \caption{
    Evolution of  \(\rho\) for two-dimensional Patlak--Keller--Segel system for Example \ref{PKS_2d_sub} (super-critical).}
    \label{fig:pks_2d_evolution2}
\end{figure}

\begin{example}\label{PME_M3}
The porous medium equation provides an example of nonlinear diffusion. We consider the following one-dimensional case to further examine the applicability of the proposed method. Specifically,
\begin{equation*}
\partial_t \rho
=
\partial_x
\left(
\rho \partial_x \rho^2
\right),
\qquad
x\in[-6, 6] ,
\quad
t>0.
\end{equation*}

The exact solution is given by the Barenblatt-type profile
\begin{equation*}
\rho(t,x)
=
(t+1)^{-\frac14}
\left(
2
-
\frac18 x^2 (t+1)^{-\frac12}
\right)_+^{\frac12}.
\end{equation*}
The solution and training loss are presented in Figure \ref{fig:PME}.
\end{example}
\begin{figure}[htbp]
    \centering
    \begin{minipage}{0.45\textwidth}
        \centering
        \includegraphics[width=\linewidth]{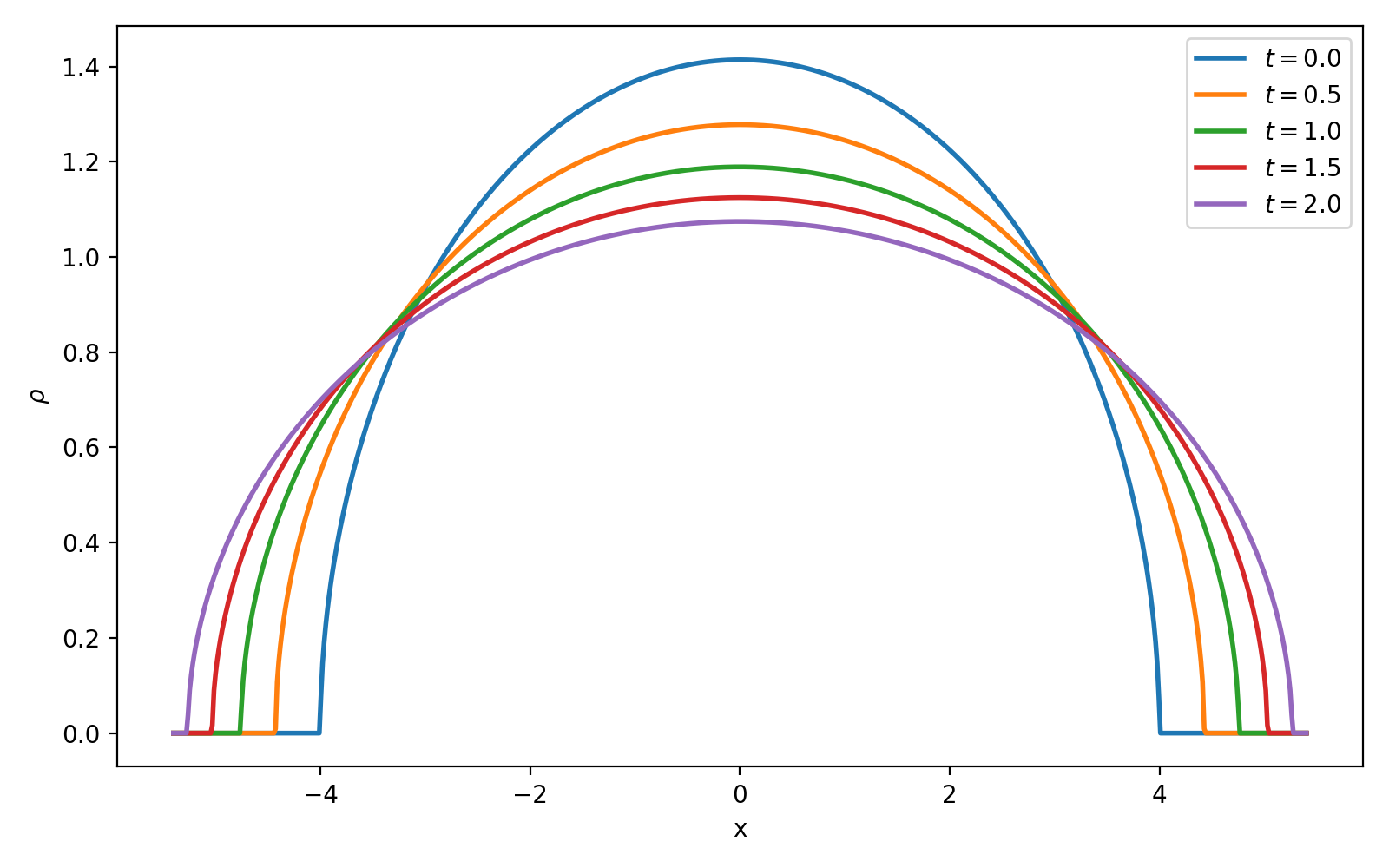}
    \end{minipage}
    \hfill
    \begin{minipage}{0.45\textwidth}
        \centering
        \includegraphics[width=\linewidth]{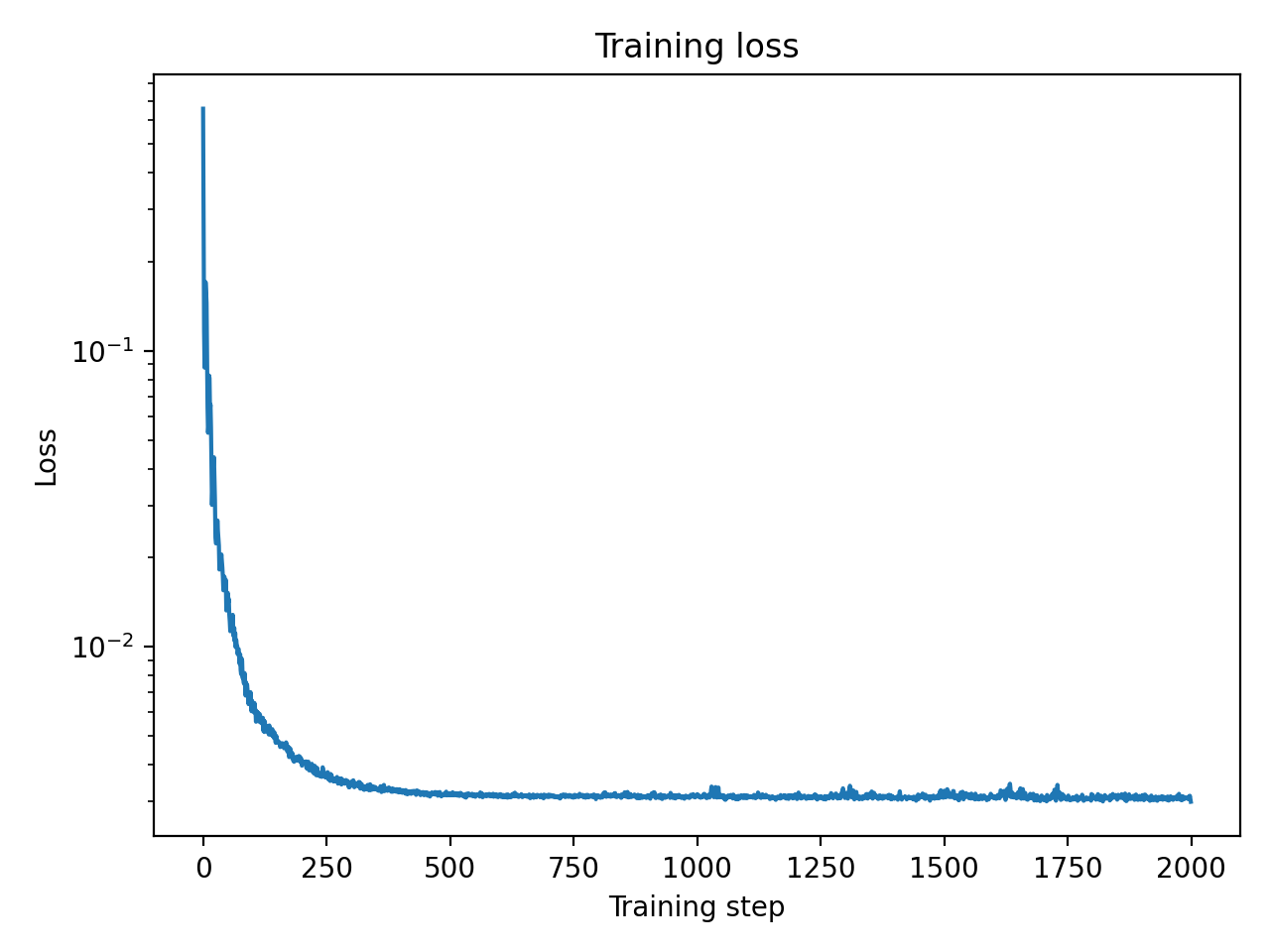}
    \end{minipage}

    \caption{
    Evolution of $\rho$ and training loss for Example \ref{PME_M3}.
    }
    \label{fig:PME}
\end{figure}

\begin{example}[3D and 10D Ornstein--Uhlenbeck] \label{ex:ou3d}
We next consider the Ornstein--Uhlenbeck problem in three and ten
dimensions to investigate the performance of the proposed method
in higher-dimensional settings. Taking
$H(\rho)=\rho\log\rho-\rho$, so that $H'(\rho)=\log\rho$,
$V(x)=\tfrac12|x|^2$, and $W\equiv0$,
equation \eqref{Ori_sys} reduces to the linear Fokker--Planck equation 
\[
  \partial_t\rho=\nabla_x\cdot(\nabla_x\rho+x\rho),
  \qquad x\in[-6, 6]^d,\ t\in[0,1],
  \qquad d=3,10.
\]
The initial density is an isotropic Gaussian
\[
  \rho_0(x)=\mathcal N(\mu_0,\sigma_0^2 I_d),
  \qquad \mu_0=1.5\,\mathbf{1}_d,\ \sigma_0=0.5.
\]
Since $V$ is quadratic, $\rho(t,\cdot)$ remains an isotropic Gaussian for all $t\ge0$:
writing $\rho(t,\cdot)=\mathcal N(\mu(t),\sigma^2(t)I_d)$, the mean and variance evolve as
\[
  \mu(t)=\mu_0e^{-t},
  \qquad
  \sigma^2(t)=1+(\sigma_0^2-1)e^{-2t},
\]
with the corresponding score and reference velocity
\[
  s(t,x)=\nabla_x\log\rho(t,x)=-\frac{x-\mu(t)}{\sigma^2(t)},
  \qquad
  v^*((t,x)=-s(t,x)-\nabla_x V(x)=\frac{x-\mu(t)}{\sigma^2(t)}-x.
\]
We take $\Delta t=0.01$, $N_p=512$ particles per batch, and $N_{\mathrm{step}}=2000$
Adam iterations with learning rate $10^{-3}$. The network consists of
four hidden layers of width $128$ with $\tanh$ activation functions.

For the three-dimensional case, because $\rho(t,\cdot)$ is isotropic, we assess accuracy through the radial
distance $r=|x-\mu(t)|$ of the recovered particles, whose exact density is
\[
  f(r, t)=\sqrt{\frac{2}{\pi}}\,\frac{r^2}{\sigma(t)^3}
        \exp\!\left(-\frac{r^2}{2\sigma^2(t)}\right).
\]
Figure~\ref{fig:ou3d} shows the training loss together with the recovered
radial distributions and the corresponding exact density $f(r, t)$ at $t=0,0.25,0.5,1$.

For the ten-dimensional case, since the ten-dimensional density cannot be visualized directly, we present
representative two-dimensional views in the coordinate pairs
$(x_1,x_2)$ and $(x_9,x_{10})$.
Figures~\ref{fig:ou10d_x1x2} and~\ref{fig:ou10d_x9x10} compare the learned
particle distributions and velocity fields with the exact solutions at
$t=0$, $0.25$, $0.5$, and $1$.

Figure~\ref{fig:dim_Np} and Table~\ref{tab:dim_Np} report the wall-clock
time per training iteration, the peak GPU memory, and the accuracy for
$d=1,\dots,15$ and $N_p\in\{128,256, 512\}$, with $\Delta t=0.01$, the network
architecture, and 2000 Adam iterations kept fixed. The time per iteration
stays between 1.2 and 1.5\,s over the whole range and shows no systematic
dependence on $d$ or $N_p$: at this problem size the GPU is not saturated,
and the cost of an iteration is dominated by the $N_t$ sequential time
steps, whereas the arithmetic work per step scales as $O(N_p d)$. The peak
memory grows linearly in both $d$ and $N_p$, approximately as
$1.6\times10^{-3}\,dN_p+1.4\times10^{-3}\,N_p$\,GB, since the computational
graph of the rollout, including $\nabla_x v_\theta$ and
$\nabla_x(\nabla_x\cdot v_\theta)$, is stored for back-propagation.
Neither the cost nor the memory exhibits the exponential growth in $d$
of grid-based solvers.

As shown in the bottom row of Figure~\ref{fig:dim_Np}, for $d\le 3$ the
relative $L^1$ and $L^2$ errors of the density are below $10^{-2}$ and
essentially independent of $N_p$. With $N_p$ fixed, the errors increase
with $d$. The relative residual of $v_\theta$ on the training particles
remains at a few percent in all dimensions, while the error on fresh
particles is up to five times larger at $N_p=128$; the dominant error in
high dimension is therefore the finite-sample generalization error rather
than the optimization or time discretization error. Increasing $N_p$
reduces this error substantially, and the reduction grows with $d$: from
$N_p=128$ to $N_p=512$, the relative $L^1$ error decreases from
$2.7\times10^{-1}$ to $9.6\times10^{-2}$ at $d=10$ and from
$7.5\times10^{-1}$ to $1.8\times10^{-1}$ at $d=15$. Since the memory grows
only linearly in $N_p$, the accuracy in high dimension can be improved at
a cost that grows only linearly.
\end{example}

\begin{figure}[htbp]
\centering
\begin{minipage}{0.4\textwidth}
  \centering
  \includegraphics[width=\linewidth]{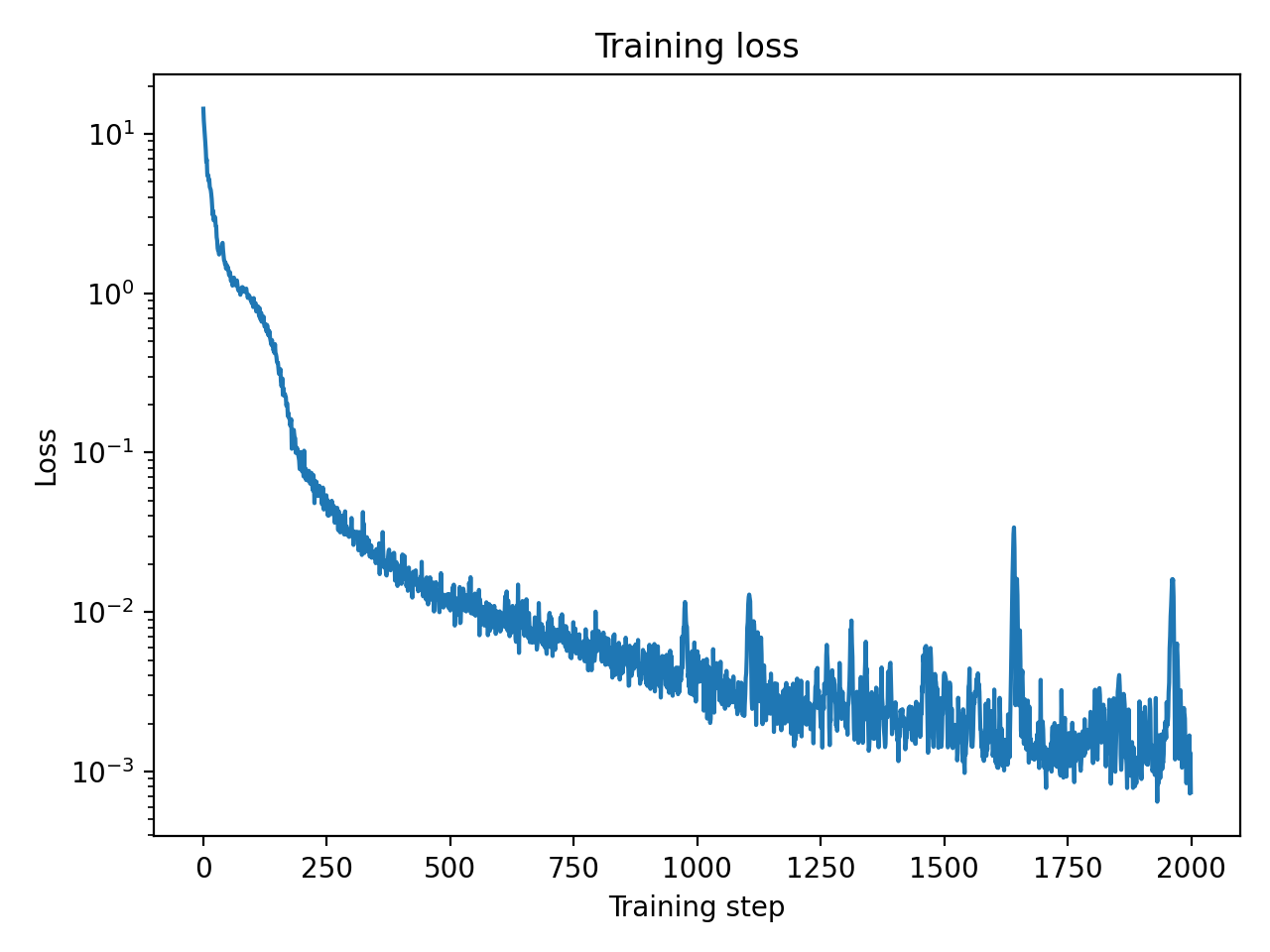}
  \subcaption{Training loss}
\end{minipage}
\begin{minipage}{0.4\textwidth}
  \centering
  \includegraphics[width=\linewidth]{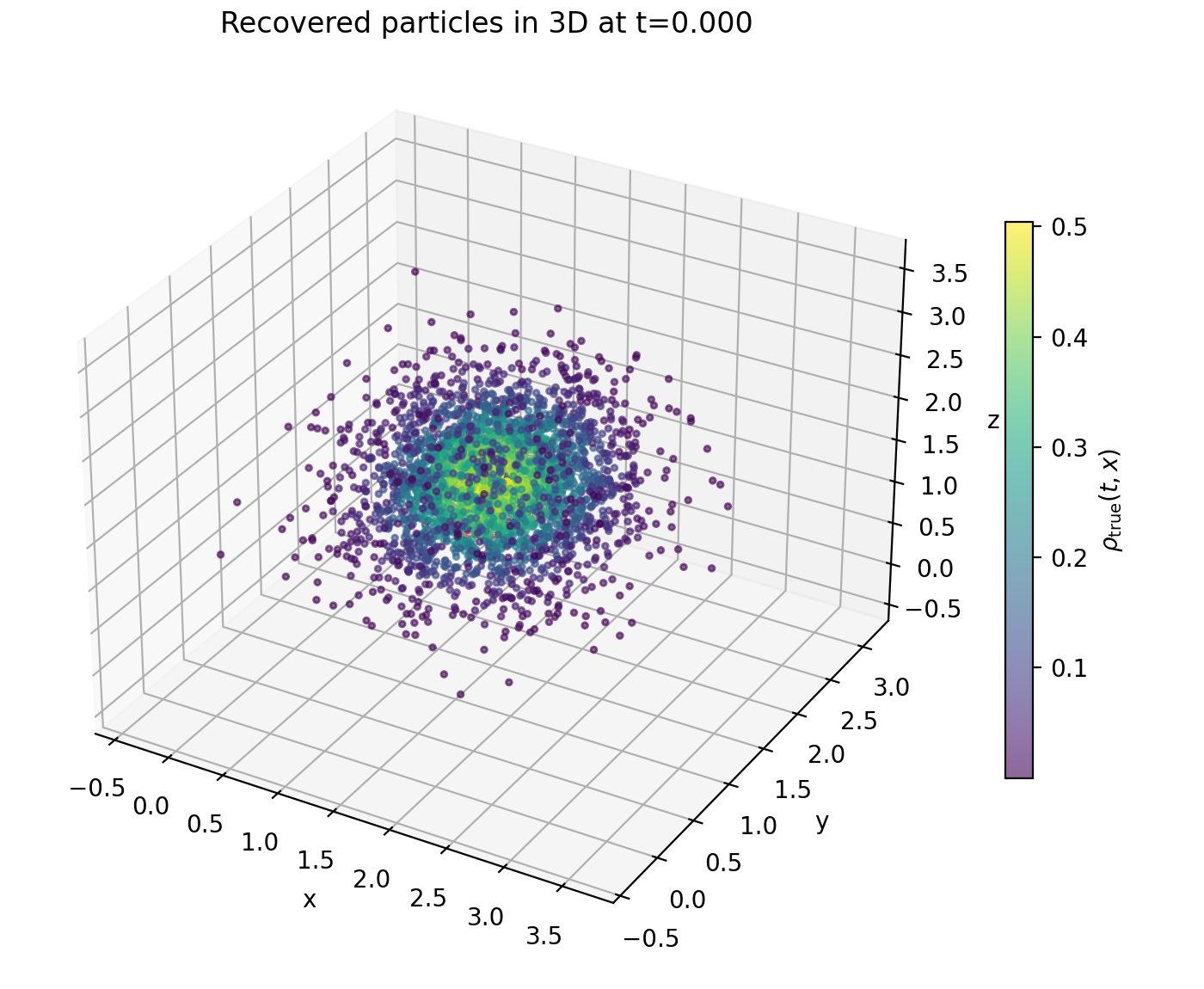}
  \subcaption{$t=0$}
\end{minipage}\\[1ex]
\begin{minipage}{0.4\textwidth}
  \centering
  \includegraphics[width=\linewidth]{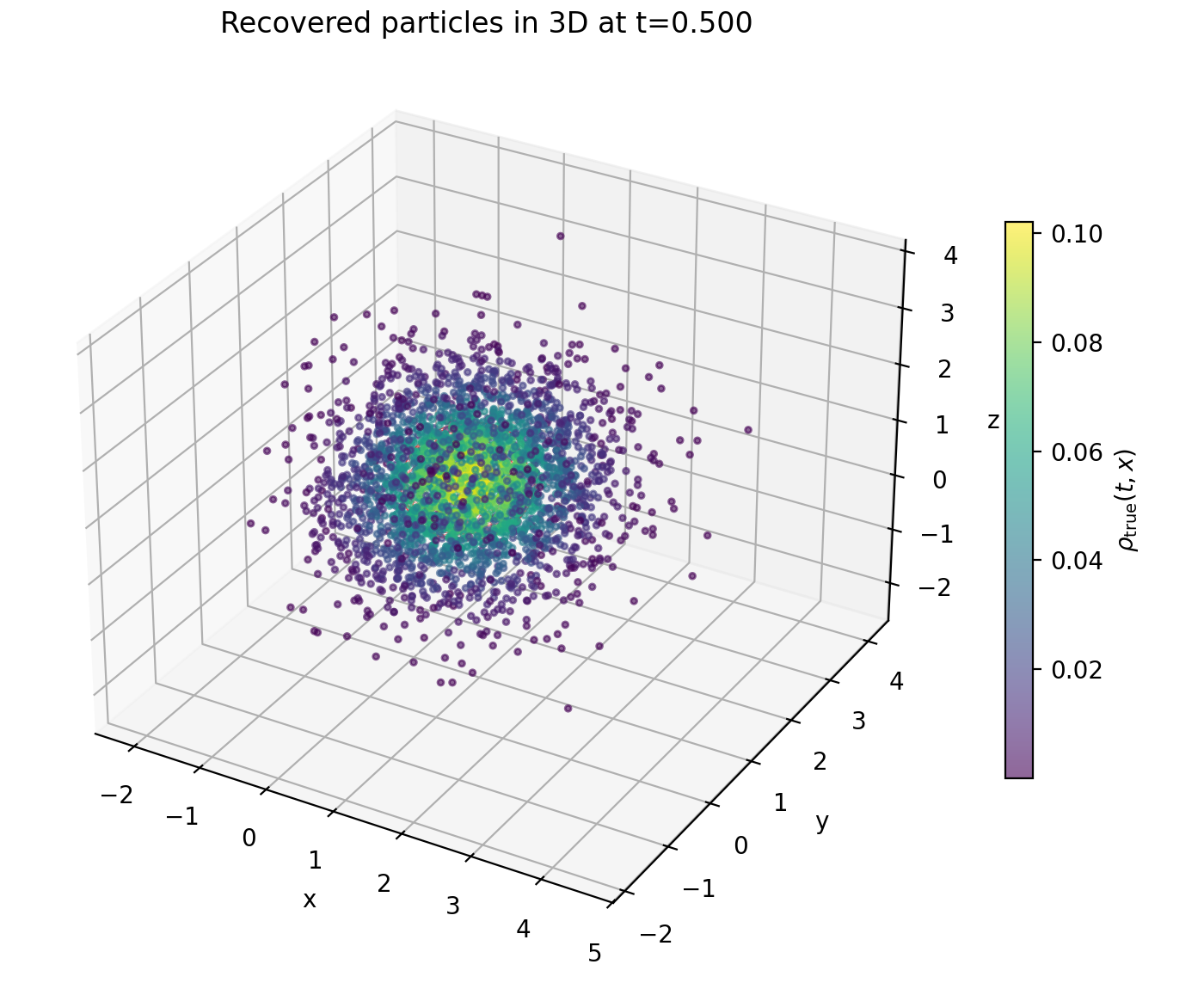}
  \subcaption{$t=0.5$}
\end{minipage}
\begin{minipage}{0.4\textwidth}
  \centering
  \includegraphics[width=\linewidth]{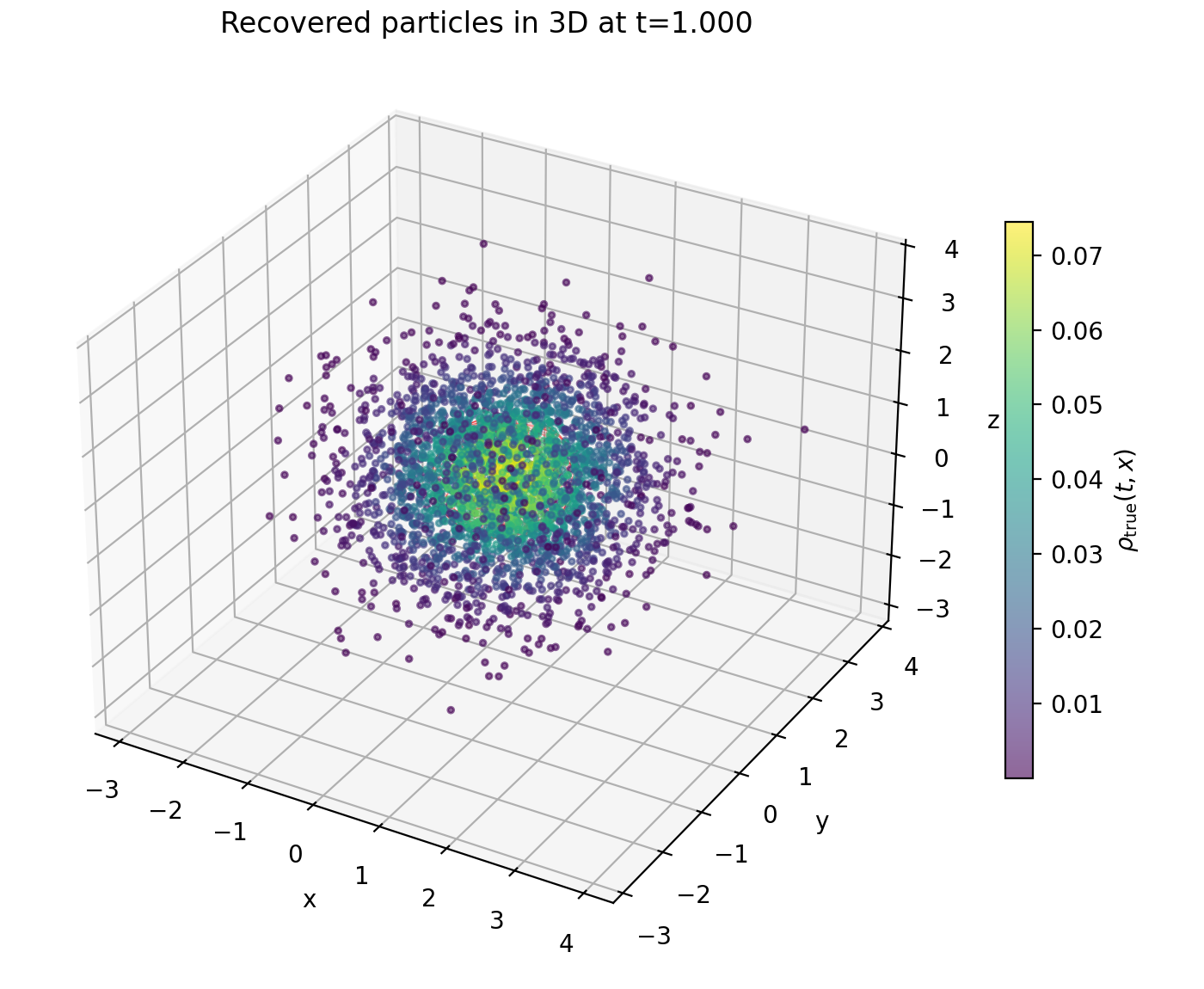}
  \subcaption{$t=1.0$}
\end{minipage}
\caption{Training loss and recovered particles at $t=0,0.5,1.0$ for the 3D Ornstein--Uhlenbeck example (Example~\ref{ex:ou3d}). }
\label{fig:ou3d}
\end{figure}

\begin{figure}[htbp]
\centering
\includegraphics[width=0.8\textwidth]{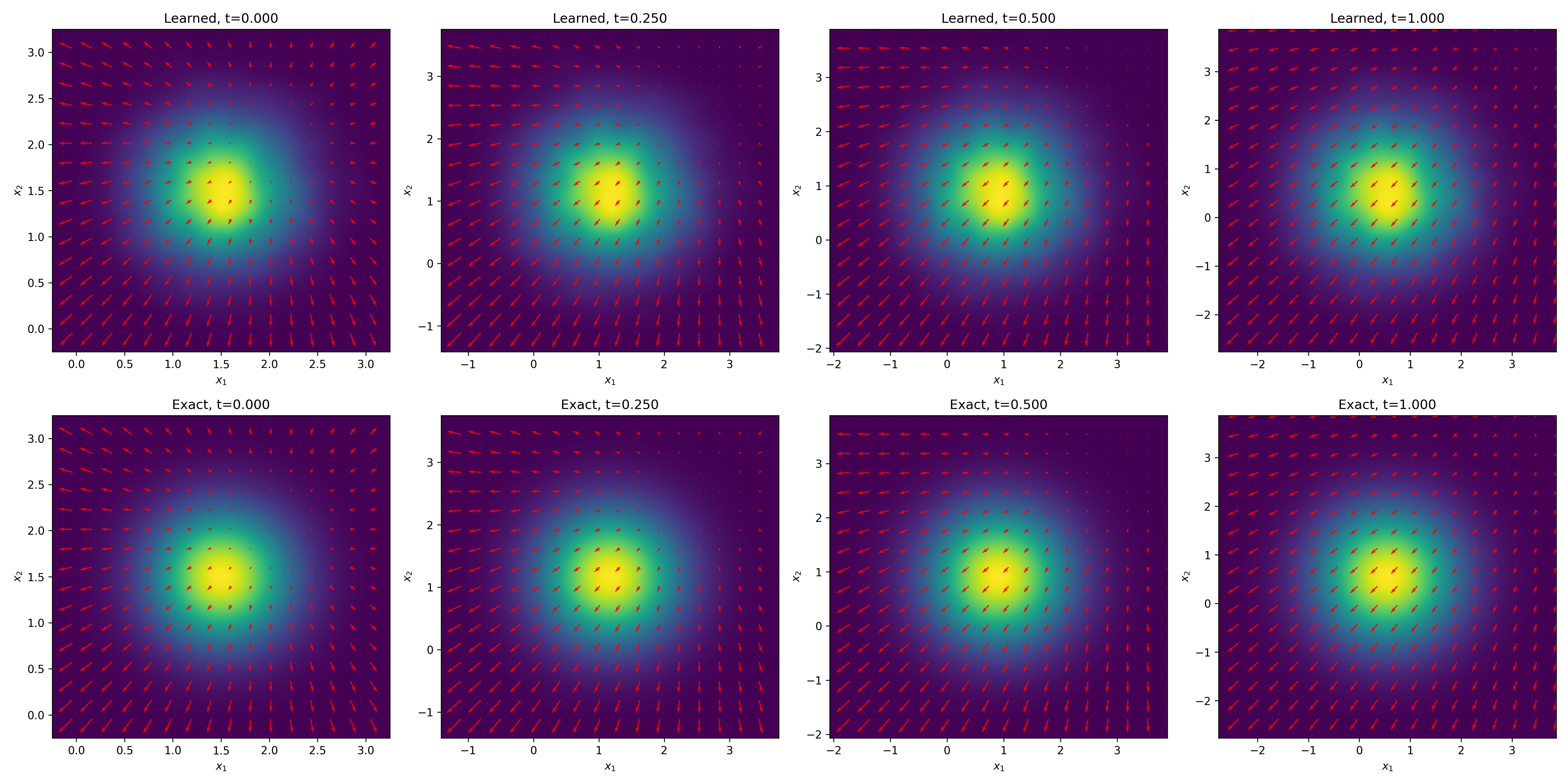}
\caption{Learned and exact dynamics in the $(x_1,x_2)$ coordinates for
Example~\ref{ex:ou3d}.}
\label{fig:ou10d_x1x2}
\end{figure}

\begin{figure}[htbp]
\centering
\includegraphics[width=0.8\textwidth]{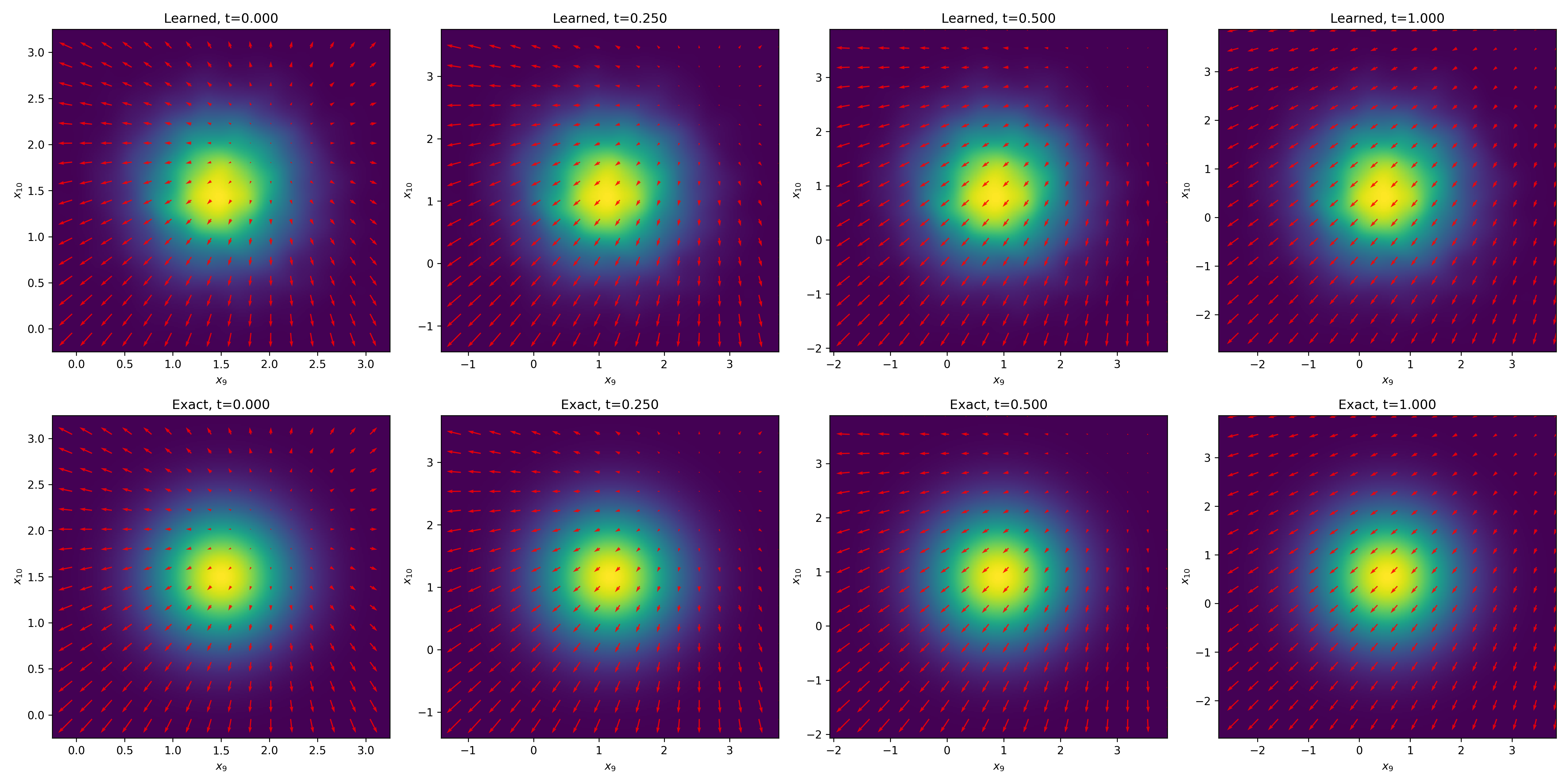}
\caption{Learned and exact dynamics in the $(x_9,x_{10})$ coordinates for
Example~\ref{ex:ou3d}.}
\label{fig:ou10d_x9x10}
\end{figure}

\begin{table}[ht]
\centering
\small
\setlength{\tabcolsep}{6pt}
\renewcommand{\arraystretch}{1.12}
\begin{tabular}{ccrrrrrr}
\toprule
$d$ & $N_p$ & time/iter (s) & memory (GB)
& $L^1$ & $L^2$ & KL & $e_v$ \\
\midrule

1
& 128 & 1.45 & 0.4
& $5.33\times10^{-3}$ & $5.19\times10^{-3}$
& $1.16\times10^{-4}$ & $1.37\times10^{-2}$ \\
& 256 & 1.34 & 1.0
& $4.29\times10^{-3}$ & $4.12\times10^{-3}$
& $3.41\times10^{-5}$ & $1.40\times10^{-2}$ \\
& 512 & 1.22 & 1.6
& $3.46\times10^{-3}$ & $3.27\times10^{-3}$
& $5.71 \times 10^{-5}$ & $1.43\times10^{-2}$ \\
\addlinespace[2pt]

2
& 128 & 1.52 & 0.6
& $5.70\times10^{-3}$ & $5.73\times10^{-3}$
& $1.97\times10^{-4}$ & $1.88\times10^{-2}$ \\
& 256 & 1.42 & 1.5
& $5.70\times10^{-3}$ & $5.62\times10^{-3}$
& $1.36\times10^{-4}$ & $1.74\times10^{-2}$ \\
& 512 & 1.32 & 2.4
& $5.71\times10^{-3}$ & $5.52\times10^{-3}$
& $9.38\times10^{-5}$ & $1.61\times10^{-2}$ \\
\addlinespace[2pt]

3
& 128 & 1.41 & 0.8
& $9.29\times10^{-3}$ & $7.58\times10^{-3}$
& $4.25\times10^{-4}$ & $2.39\times10^{-2}$ \\
& 256 & 1.34 & 2.0
& $8.81\times10^{-3}$ & $7.51\times10^{-3}$
& $3.84\times10^{-4}$ & $2.07\times10^{-2}$ \\
& 512 & 1.27 & 3.2
& $8.35\times10^{-3}$ & $7.44\times10^{-3}$
& $3.46\times10^{-4}$ & $1.80\times10^{-2}$ \\
\addlinespace[2pt]

5
& 128 & 1.38 & 1.2
& $3.42\times10^{-2}$ & $2.69\times10^{-2}$
& $3.01\times10^{-3}$ & $3.58\times10^{-2}$ \\
& 256 & 1.33 & 3.0
& $3.24\times10^{-2}$ & $2.72\times10^{-2}$
& $2.10\times10^{-3}$ & $2.86\times10^{-2}$ \\
& 512 & 1.28 & 4.8
& $3.06\times10^{-2}$ & $2.75\times10^{-2}$
& $1.47\times10^{-3}$ & $2.29\times10^{-2}$ \\
\addlinespace[2pt]

10
& 128 & 1.19 & 2.2
& $2.66\times10^{-1}$ & $2.66\times10^{-1}$
& $7.30\times10^{-2}$ & $8.59\times10^{-2}$ \\
& 256 & 1.30 & 5.55
& $1.60\times10^{-1}$ & $1.47\times10^{-1}$
& $2.68\times10^{-2}$ & $6.13\times10^{-2}$ \\
& 512 & 1.40 & 8.9
& $9.57\times10^{-2}$ & $8.07\times10^{-2}$
& $9.85\times10^{-3}$ & $4.38\times10^{-2}$ \\
\addlinespace[2pt]

15
& 128 & 1.34 & 3.3
& $7.48\times10^{-1}$ & $1.85\times10^{0}$
& $4.53\times10^{-1}$ & $1.54\times10^{-1}$ \\
& 256 & 1.40 & 8.15
& $3.65\times10^{-1}$ & $6.25\times10^{-1}$
& $1.20\times10^{-1}$ & $9.60\times10^{-2}$ \\
& 512 & 1.45 & 13.0
& $1.78\times10^{-1}$ & $2.11\times10^{-1}$
& $3.15\times10^{-2}$ & $5.98\times10^{-2}$ \\

\bottomrule
\end{tabular}

\caption{Cost, memory, and accuracy versus $d$ and $N_p$ for the OU example
($\Delta t=0.01$, 2000 iterations). $L^1$, $L^2$: relative density errors
at $t=T$; KL: $\mathrm{KL}(\tilde\rho_T\|\rho_T)$; $e_v$: relative velocity
error. KL values below $10^{-5}$ are within Monte Carlo error.}
\label{tab:dim_Np}
\end{table}

\begin{figure}[ht]
\centering
\includegraphics[width=0.9\textwidth]{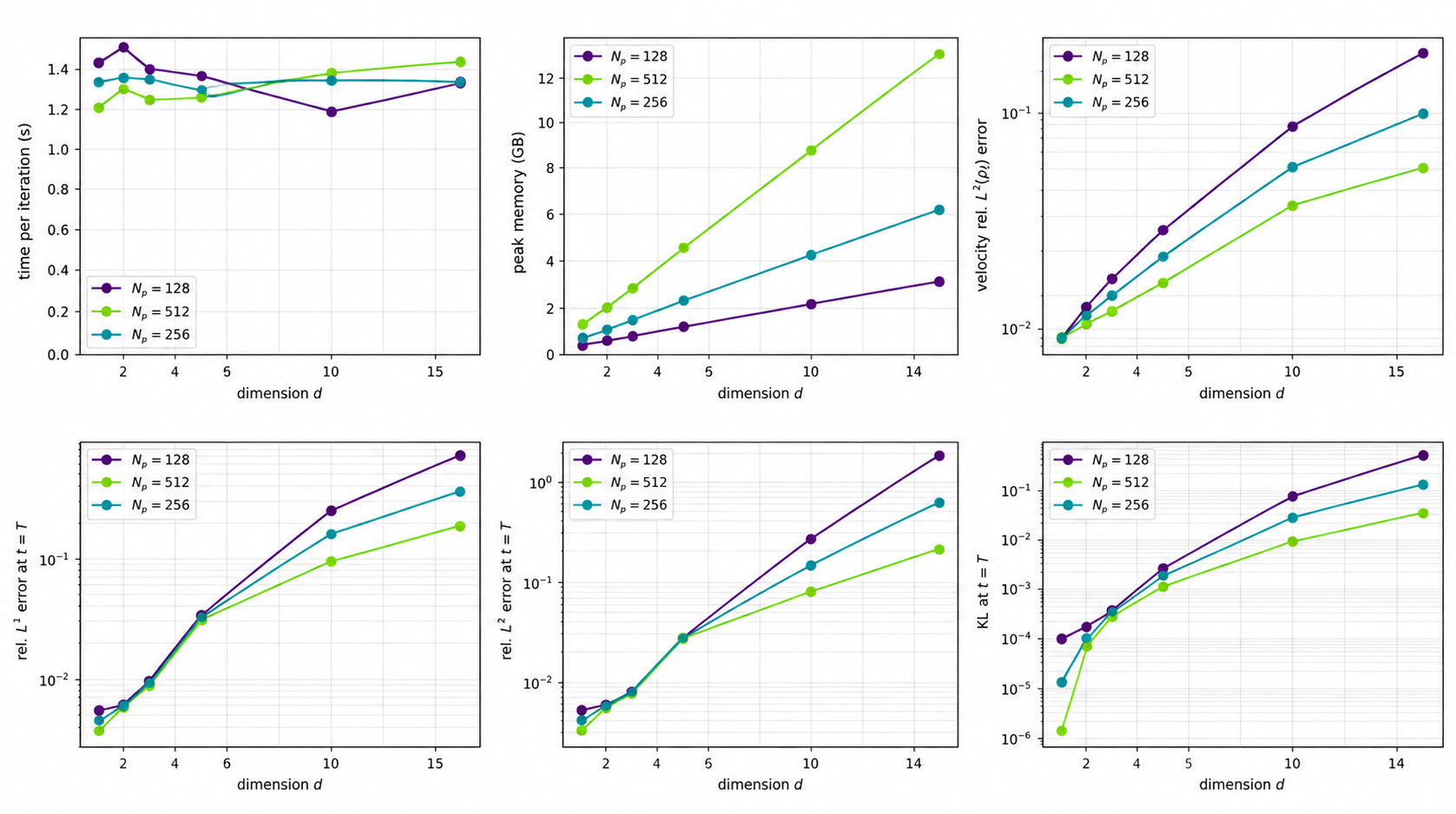}
\caption{Cost, memory, and accuracy versus dimension $d$ for
$N_p\in\{128,256,512\}$ (OU example, $\Delta t=0.01$, 2000 iterations).
Top: time per iteration, peak GPU memory, and relative velocity error
$e_v$. Bottom: relative $L^1$ and $L^2$ errors and KL divergence of the
density at $t=T$.}
\label{fig:dim_Np}
\end{figure}

\section{Conclusion and Outlook}

In this work, we developed a variational neural transport framework for nonlinear aggregation-diffusion equations based on the Onsager variational principle. By reformulating the governing PDE as a constrained optimization problem, the transport velocity naturally emerges as the optimization variable, leading to a variational learning objective that reflects the underlying free-energy dissipation. The resulting particle-based method parameterizes the velocity field with a neural network and evaluates the variational objective along characteristic trajectories. This mesh-free formulation  avoids explicit spatial discretization and provides a flexible and scalable approach to high-dimensional problems.

On the theoretical side, we established the equivalence between the proposed variational formulation and the original aggregation--diffusion equation, derived stability estimates based on the KL divergence, and analyzed the evolution of the density and generalized score function along characteristics, which provides a foundation for particle-based implementation. We further developed a rigorous error analysis that separates neural approximation, optimization, temporal discretization, and Monte Carlo sampling errors. Numerical experiments on representative models, including the Fokker-Planck, Doi-Onsager, Patlak-Keller-Segel, and porous medium equations, demonstrate the accuracy, robustness, and flexibility of the framework. 

Several directions remain open. These include extensions to coupled PDE systems, more general mobility functions, anisotropic transport, and non-Euclidean variational structures, as well as adaptive strategies for efficient training over long-time dynamics.
More broadly, the framework suggests that variational principles such as the Onsager principle can provide a systematic pathway for constructing variational neural solvers beyond aggregation--diffusion equations, with potential applications to a broader class of gradient-flow and multiphysics systems where preserving underlying variational and physical structures is essential.
\bibliographystyle{amsplain}
\bibliography{ref}

\appendix

\renewcommand{\theequation}{A.\arabic{equation}}
\setcounter{equation}{0}

\section{Technical Proofs}\label{AppendixA}
Proof of Proposition \ref{KL_sta}:

(i) Positivity of $E(\rho \mid \pi)$. Since $H$ is convex, we have
\[
H(\rho)-H(\pi)-H'(\pi)(\rho-\pi)\ge0.
\]
Integrating over $\Omega$ gives
\[
\int_\Omega
\bigl(H(\rho)-H(\pi)-H'(\pi)(\rho-\pi)\bigr)\,dx
\ge0.
\]
Moreover, since $W$ is positive semi-definite, namely,
\[
\int_{\Omega\times\Omega}
W(x-y)f(x)f(y)\,dx\,dy\ge0
\]
for every admissible function $f$. Taking $f=\rho-\pi$, we obtain
\[
\frac12
\int_{\Omega\times\Omega}
W(x-y)(\rho-\pi)(x)(\rho-\pi)(y)\,dx\,dy
\ge0.
\]
Hence,
\[
E(\rho\mid\pi)\ge0.
\]

(ii) Non--increasing. Let $\eta=\rho-\pi$. We first show that
\[
E(\rho\mid\pi)=E(\rho)-E(\pi).
\]
Indeed, writing $\rho=\pi+\eta$, we have
\[
\rho(x)\rho(y)-\pi(x)\pi(y)
=
\pi(x)\eta(y)+\eta(x)\pi(y)+\eta(x)\eta(y).
\]
Using the symmetry of $W$, the interaction-energy difference becomes
\[
\begin{aligned}
&\frac12
\int_{\Omega\times\Omega}
W(x-y)
\bigl(\rho(x)\rho(y)-\pi(x)\pi(y)\bigr)
\,dx\,dy  \\
&\qquad
=
\int_\Omega
(W*\pi)(x)\eta(x)\,dx
+
\frac12
\int_{\Omega\times\Omega}
W(x-y)\eta(x)\eta(y)\,dx\,dy.
\end{aligned}
\]
Therefore,
\[
\begin{aligned}
E(\rho)-E(\pi)
={}&
E(\rho\mid\pi)
+
\int_\Omega
\bigl(H'(\pi)+V+W*\pi\bigr)(\rho-\pi)\,dx.
\end{aligned}
\]
Since $\pi$ is a stationary solution,
\[
\mu[\pi]
=
H'(\pi)+V+W*\pi
=
C
\]
for some constant $C$. Together with the mass conservation $\int_\Omega(\rho-\pi)\,dx=0,$ we obtain
\[
E(\rho\mid\pi)=E(\rho)-E(\pi).
\]

Since $\pi$ is independent of time,
\[
\frac{d}{dt}E(\rho\mid\pi)
=
\frac{d}{dt}E(\rho) = 
\int_\Omega
\mu[\rho]\,
\nabla_x\cdot
\bigl(\rho\nabla_x\mu[\rho]\bigr)\,dx.
\]
Using integration by parts and the no-flux boundary condition yields
\[
\frac{d}{dt}E(\rho)
=
-\int_\Omega
\rho
|\nabla_x\mu[\rho]|^2
\,dx.
\]
Consequently,
\[
\frac{d}{dt}E(\rho\mid\pi)
=
-\int_\Omega
\rho
|\nabla_x\mu[\rho]|^2
\,dx,
\]
which shows that $E(\rho\mid\pi)$ is monotonically non-increasing along solutions and thus serves as a Lyapunov functional for the system.. This completes the proof.

Proof of Lemma \ref{L2_KL}:
    Let $r = \frac{\rho_1}{\rho_2}$, we have
    \begin{equation*}
        \rho_1 - \rho_2 = \rho_2(r-1).
    \end{equation*}
    So
    \begin{equation}\label{ineq_M}
        |\rho_1 - \rho_2|^2  = (\rho_2)^2 (r-1)^2 \leq M^* \rho_2 (r-1)^2.
    \end{equation}
    And by the bound of $\rho_1,\rho_2$, we can get the bound of $r$:
    \begin{equation*}
        0 \leq r \leq \frac{M^*}{m^*}.
    \end{equation*}
    Define function
    \begin{equation*}
        g(r)= \frac{(r-1)^2}{r\log r - r + 1}, \quad r\in (0, \infty), r \neq 1. 
    \end{equation*}
Note that $g(r)$  satisfies 
$$
\lim_{r\rightarrow 1} g(r) = 2.
$$
By defining $g(0) = 1$ and $g(1) = 2$, $g$ can be continuously extended to $I:=\left[\frac{m^*}{M^*}, \frac{M^*}{m^*}\right]$ and so $g(r)$ is bounded on interval $I$.  Hence there exists a constant $C_* > 0$, depending on $m^*, M^*$, such that 
    \begin{equation*}
        (r-1)^2 \leq C_* (r\log r - r + 1).
    \end{equation*}
    Plug into (\ref{ineq_M}) and take the integral, we have
    \begin{equation*}
        \begin{aligned}
            \| \rho_1 - \rho_2\|_{L^2(\Omega)}^2 = \int_\Omega |\rho_1 - \rho_2|^2 \,dx & \leq  C_* \int_\Omega \rho_2 (r\log r - r + 1) \,dx\\
            & =  C_* \bigg[ \int_\Omega \rho_1 \log \frac{\rho_1}{\rho_2} \,dx - \int_\Omega \rho_1\,dx  + \int_\Omega \rho_2 \,dx \bigg]\\
            & =  C_* D_{KL}(\rho_1\|\rho_2).
        \end{aligned}
    \end{equation*}    

Proof of Lemma \ref{Lemma_KL_t}:
We compute
    \begin{equation*}
\begin{aligned}
    \frac{d}{dt}D_{KL}(\rho\|\rho^*)  & = \int_{\Omega}\partial_t \rho  \log\frac{\rho}{\rho^*} - \frac{\rho}{\rho^*} \partial_t \rho^*\,dx.
\end{aligned}
\end{equation*}

Using integration by parts, and consider the no-flux boundary condition of $\rho$ and $\rho^*$, we have:
\begin{equation*}
    \begin{aligned}
        \frac{d}{dt}D_{KL}(\rho\|\rho^*)  & = \int_{\Omega} - \nabla_x \cdot( \rho v) \log\frac{\rho}{\rho^*} +  \frac{\rho}{\rho^*} \nabla_x \cdot (\rho^* v^*)\,dx\\
    & = \int_\Omega  \rho v \nabla_x\log \frac{\rho}{\rho^*} - \nabla_x(\frac{\rho}{\rho^*}) \rho^* v^*\,dx\\ 
    & = \int_\Omega  \rho v \nabla_x\log \frac{\rho}{\rho^*} - \nabla_x( \log \frac{\rho}{\rho^*}) \rho v^*\,dx \\
    &  = \int_\Omega  \rho\nabla_x\log \frac{\rho}{\rho^*} \cdot  (v - v^*) \,dx.
    \end{aligned}
\end{equation*}

Proof of Theorem \ref{KL_control}: 
The proof proceeds in several steps: \\
Step 1:  Decomposition for control.   
    \begin{equation*}
\begin{aligned}
\frac{d}{dt}D_{KL}(\rho\|\rho^*)
&=
\int_\Omega \rho \nabla_x \log \frac{\rho}{\rho^*} \cdot (v-v^*) \,dx \\
&=
\underbrace{
\int_\Omega \rho \nabla_x \log \frac{\rho}{\rho^*}
\cdot
\Big(
v+\nabla_x\log\rho+\nabla_xV+\nabla_x(W*\rho)
\Big)\,dx
}_{\mathrm{I}}
\\
&\quad
-\underbrace{
\int_\Omega
\rho
\left|
\nabla_x \log \frac{\rho}{\rho^*}
\right|^2
dx
}_{\mathrm{II}}
\\
&\quad
-\underbrace{
\int_\Omega
\rho
\nabla_x \log \frac{\rho}{\rho^*}
\cdot
\nabla_x\big(W*(\rho-\rho^*)\big)
\,dx
}_{\mathrm{III}}.
\end{aligned}
\end{equation*}
Applying Young's inequality, we estimate 
\begin{equation*}
\mathrm{I}
\leq
\frac12
\int_\Omega
\rho
\left|
\nabla_x \log \frac{\rho}{\rho^*}
\right|^2
dx
+
\frac12
\int_\Omega
\rho
\left|
v+\nabla_x\log\rho+\nabla_xV+\nabla_x(W*\rho)
\right|^2
dx,
\end{equation*}
and
\begin{equation*}
\mathrm{III}
\leq
\frac12
\int_\Omega
\rho
\left|
\nabla_x \log \frac{\rho}{\rho^*}
\right|^2
dx
+
\frac12
\int_\Omega
\rho
\left|
\nabla_x\big(W*(\rho-\rho^*)\big)
\right|^2
dx.
\end{equation*}
Substituting these estimates into the decomposition identity and canceling the common Fisher information yields 
\begin{equation*}
\frac{d}{dt}D_{KL}(\rho\|\rho^*)
\leq
\frac12
\int_\Omega
\rho
\left|
v+\nabla_x\log\rho+\nabla_xV+\nabla_x(W*\rho)
\right|^2
dx
+
\frac12
\int_\Omega
\rho
\left|
\nabla_x\big(W*(\rho-\rho^*)\big)
\right|^2
dx.
\end{equation*}
Step 2. Bounding both $\rho$ and $\rho^*$. \\
Since $v$ is Lipschitz in $x$ and continuous in $t$, its divergence satisfies
$$
\nabla_x\cdot v\in L^\infty([0,T]\times\Omega).
$$
By the method of characteristics,
\begin{equation*}
\rho(t,X(t;x_0))
=
\rho_0(x_0)
\exp\left(
-\int_0^t
\nabla_x\cdot v(s,X(s;x_0))\,ds
\right),
\end{equation*} 
we obtain,  with $C_v:=\|\nabla_x\cdot v\|_{L^\infty([0,T]\times\Omega)}$, 
$$
m_0 e^{-C_vT} \le
\rho(t,x)
\le M_0
e^{C_vT}, 
$$
uniformly for $(t,x)\in[0,T]\times\Omega)$. Next, we derive uniform bounds for  $\rho^*$, the classical solution of  (\ref{Ori_sys}). By the strong maximum principle; see \cite[Chapter~2, Theorem~2.9]{Lieberman1996}, we have 
$$
\rho^*(t,x)>0, \quad (t, x)\in (0, T]\times \Omega. 
$$
Moreover, since $\rho^*$ is continuous on the compact set $[0,T]\times\overline{\Omega}$, it attains its minimum and maximum there. Consequently, $\rho^*$ is uniformly bounded away from zero and above.
  
Combining this with the estimate established for $\rho$, we conclude that there exist positive constants $m_*,M_*>0$ such that
\begin{equation*}
m_*
\le
\rho^*(t,x), \rho(x, t)
\le
M_*,
\qquad
(t,x)\in[0,T]\times\Omega.
\end{equation*}

Step 3. Bounding the interaction term. \\
Using the uniform bound $\rho\le M^*$, we obtain
\begin{align*} 
\int_\Omega
\rho
\left|
\nabla_x \left(W*(\rho-\rho^*)\right)
\right|^2 dx
& \le M^*
\left\|\nabla_x W*(\rho-\rho^*)
\right\|_{L^2(\Omega)}^2 \\
& \le M^*
\|\nabla_x W\|^2_{L^1(\Omega)}
\|\rho-\rho^*\|^2_{L^2(\Omega)},
\end{align*} 
where we have used $W\in C^2(\Omega)$ and Young's convolution inequality of form 
$$
\|w*g\|_{L^2}\leq \|w\|_{L^1}\|g\|_{L^2}.
$$

Note that mass conservation for both $\rho$ and $\rho^*$ yields 
$$
\int_{\Omega}\rho(x, t)dx=\int_{\Omega}\rho^*(x, t)dx=\int_{\Omega}\rho_0(x)dx.
$$
Applying Lemma \ref{L2_KL}, we obtain  
\begin{equation*}
\|\rho-\rho^*\|_{L^2(\Omega)}^2
\leq
M^*C_* D_{KL}(\rho\|\rho^*).
\end{equation*}
Hence 
$$
\int_\Omega
\rho
\left|
\nabla_x \left(W*(\rho-\rho^*)\right)
\right|^2 dx \leq 2 C D_{KL}(\rho\|\rho^*),
$$
where 
\begin{equation*}
C
=
\frac12
(M^*)^2
C_*
\|\nabla_xW\|_{L^1(\Omega)}^2.
\end{equation*}
Step 4. Control of $D_{KL}$. \\ 
Combining the above estimates, we obtain 
\begin{equation*}
\frac{d}{dt}D_{KL}(\rho\|\rho^*)
\leq
\int_\Omega
\rho
\left|
v+\nabla_x\log\rho+\nabla_xV+\nabla_x(W*\rho)
\right|^2
dx
+C
D_{KL}(\rho\|\rho^*).
\end{equation*}
Applying Gr\"onwall's inequality on $[0,T]$ and using the fact that $ D_{KL}(\rho(0,\cdot)\|\rho^*(0,\cdot))=0$,  we obtain
\begin{equation}
D_{KL}(\rho\|\rho^*)(t)
\le
e^{CT}
\int_0^T
\left(
\int_\Omega \rho| v + \nabla_x \log \rho + \nabla_x V + \nabla_x( W \ast \rho)|^2 \,dx
\right) dt.
\end{equation}
This completes the proof. 

Proof of Lemma \ref{Lemma_LD}:
Since the exact solution satisfies
\begin{equation*}
v^*+\nabla \mu^*=0,
\end{equation*}
the loss functional can be rewritten as  
\begin{equation*}
L_S(\theta_S)
=
\int_0^T\int_\Omega
\left|
v_{\theta_S}-v^*
+
\nabla(\mu_{\theta_S}-\mu^*)
\right|^2
\rho_{\theta_S}\,dxdt,
\end{equation*}
where
\begin{equation*}
\mu_{\theta_S}
=
H'(\rho_{\theta_S})+ V +W*\rho_{\theta_S},
\quad \text{and}\quad  
\mu^*
=
H'(\rho^*)+ V + W*\rho^*.
\end{equation*}
Therefore,
$$
\nabla(\mu_{\theta_S}-\mu^*)= \nabla\bigl(H'(\rho_{\theta_S})-H'(\rho^*)\bigr)
+\nabla W*(\rho_{\theta_S}-\rho^*),
$$
since the potential V cancels in the difference.

Using the elementary inequality $|a+b+c|^2 \leq 3(|a|^2+|b|^2+|c|^2)$, we obtain 
\begin{equation*}
\begin{aligned}
L_S(\theta_S)
&\le
3\int_0^T\int_\Omega
|v_{\theta_S}-v^*|^2
\rho_{\theta_S}\,dxdt
+
3\int_0^T\int_\Omega
\left|
\nabla\bigl(H'(\rho_{\theta_S})-H'(\rho^*)\bigr)
\right|^2
\rho_{\theta_S}\,dxdt \\
&\quad+
3\int_0^T\int_\Omega
\left|
\nabla W*(\rho_{\theta_S}-\rho^*)
\right|^2
\rho_{\theta_S}\,dxdt .
\end{aligned}
\end{equation*}

When $H'(\rho)=\log\rho,$ we have
\begin{equation*}
\nabla\bigl(H'(\rho_{\theta_S})-H'(\rho^*)\bigr)
=
\nabla\log\rho_{\theta_S}
-
\nabla\log\rho^* := s_{\theta_S} - s^*,
\end{equation*}
the difference of the corresponding score functions. 

Using Young's convolution inequality, we have 
\begin{equation*}
\|\nabla W * (\rho_{\theta_S}-\rho^*)\|_{L^2(\Omega)}
\le
\|\nabla W\|_{L^1(\Omega)}
\|\rho_{\theta_S}-\rho^*\|_{L^2(\Omega)}.
\end{equation*}
Therefore, 
\begin{equation}\label{LD}
\begin{aligned}
L_S(\theta_S)
& \le
3 \int_0^T\int_\Omega
\left|e_v(t, x)
\right|^2 \rho_{\theta_S}
dxdt +
3\int_0^T\int_\Omega
\left|e_s(t, x) 
\right|^2 \rho_{\theta_S}
dxdt \\
& \qquad +
3 \|\rho_{\theta_S}\|_{L^\infty(\Omega_T)}  \|\nabla W\|_{L^1(\Omega)}^2
\int_0^T
\|e_\rho(t, \cdot)\|_{L^2(\Omega)}^2
dt,
\end{aligned}
\end{equation}
where $\Omega_T:=[0, T]\times \Omega$ and  
\[
e_s=s_{\theta_S}-s^*,\qquad e_v=v_{\theta_S}-v^*, \quad e_\rho= 
\rho_{\theta_S}-\rho^*.
\]

We next estimate the score error involving $|e_s|^2\rho_{\theta_S}$. 
The score error satisfies
\[
\partial_t e_s
+(v_{\theta_S}\cdot\nabla)e_s
+(\nabla v_{\theta_S})^T e_s
=
-(e_v\cdot\nabla)s^*
-(\nabla e_v)^Ts^*
-\nabla(\nabla\cdot e_v),
\]
and the density satisfies 
\[
\partial_t\rho_{\theta_S}
+
\nabla\cdot(\rho_{\theta_S}v_{\theta_S})=0.
\]
Using these two equations, we evaluate the following quantities:
\[
\begin{aligned}
\frac12\frac{d}{dt}
\int_\Omega |e_s|^2\rho_{\theta_S}\,dx
&=
\int_\Omega \rho_{\theta_S} e_s\cdot\partial_t e_s\,dx
+
\frac12\int_\Omega |e_s|^2\partial_t\rho_{\theta_S}\,dx .
\end{aligned}
\]
The last term reduces to 
$$
\frac12\int_\Omega |e_s|^2\partial_t\rho_{\theta_S}\,dx
=\int_\Omega \rho_{\theta_S} e_s\cdot(v_{\theta_S}\cdot\nabla)e_s\,dx,
$$
after using the density equation and integration by parts with the zero-flux condition
\[
\rho_{\theta_S}v_{\theta_S}\cdot \mathbf{n}=0
\qquad\text{on }\partial\Omega.
\]
Hence
\[
\begin{aligned}
\frac12\frac{d}{dt}
\int_\Omega |e_s|^2\rho_{\theta_S}\,dx
=&
-\int_\Omega \rho_{\theta_S} e_s\cdot(\nabla v_{\theta_S})^T e_s\,dx\\
&-\int_\Omega \rho_{\theta_S} e_s\cdot(e_v\cdot\nabla)s^*\,dx\\
&-\int_\Omega \rho_{\theta_S} e_s\cdot(\nabla e_v)^Ts^*\,dx\\
&-\int_\Omega \rho_{\theta_S} e_s\cdot\nabla(\nabla\cdot e_v)\,dx .
\end{aligned}
\]
For the  term involving $(\nabla v_{\theta_S})^T e_s$, we have 
\[
\left|
\int_\Omega \rho_{\theta_S} e_s\cdot(\nabla v_{\theta_S})^T e_s\,dx
\right|
\le
\|\nabla v_{\theta_S}\|_{L^\infty}
\int_\Omega |e_s|^2\rho_{\theta_S}\,dx .
\]
Foe the term involving $(e_v\cdot\nabla)s^*$, we obtain 
\[
\begin{aligned}
\left|
\int_\Omega \rho_{\theta_S} e_s\cdot(e_v\cdot\nabla)s^*\,dx
\right|
&\le
\|\nabla s^*\|_{L^\infty}
\int_\Omega |e_s||e_v|\rho_{\theta_S}\,dx \\
&\le
\frac12\int_\Omega |e_s|^2\rho_{\theta_S}\,dx
+
\frac12\|\nabla s^*\|_{L^\infty}^2
\int_\Omega |e_v|^2\rho_{\theta_S}\,dx .
\end{aligned}
\]
Similarly, for the term involving $(\nabla e_v)^Ts^*$, we have 
\[
\begin{aligned}
\left|
\int_\Omega \rho_{\theta_S} e_s\cdot(\nabla e_v)^Ts^*\,dx
\right|
&\le
\|s^*\|_{L^\infty}
\int_\Omega |e_s||\nabla e_v|\rho_{\theta_S}\,dx \\
&\le
\frac12\int_\Omega |e_s|^2\rho_{\theta_S}\,dx
+
\frac12\|s^*\|_{L^\infty}^2
\int_\Omega |\nabla e_v|^2\rho_{\theta_S}\,dx .
\end{aligned}
\]
Finally, for the term involving $\nabla(\nabla\cdot e_v)$, we have 
\[
\begin{aligned}
\left|
\int_\Omega \rho_{\theta_S} e_s\cdot\nabla(\nabla\cdot e_v)\,dx
\right|
&\le
\int_\Omega |e_s||\nabla(\nabla\cdot e_v)|\rho_{\theta_S}\,dx \\
&\le
\frac12\int_\Omega |e_s|^2\rho_{\theta_S}\,dx
+
\frac12\int_\Omega |\nabla(\nabla\cdot e_v)|^2\rho_{\theta_S}\,dx .
\end{aligned}
\]

Combining the above estimates yields
\[
\begin{aligned}
\frac{d}{dt}
\int_\Omega |e_s|^2\rho_{\theta_S}\,dx
\le\;&
C_1\int_\Omega |e_s|^2\rho_{\theta_S}\,dx + 
C_1\int_\Omega
\left(
|e_v|^2
+
|\nabla e_v|^2
+
|\nabla(\nabla\cdot e_v)|^2
\right)\rho_{\theta_S}\,dx,
\end{aligned}
\]
where $$
C_1 = \max\Big\{C_v + \frac32, \frac12(1 + \|s^*\|_{L^{\infty}}^2 + \|\nabla_x s^*\|_{L^{\infty}}^2)\Big\}.
$$
By Gronwall's inequality, for \(0\le t\le T\), with \(e_s(0)=0\), we have
\begin{equation} \label{es} 
\int_\Omega |e_s(t)|^2\rho_{\theta_S}(t)\,dx
\le
C_1(T)\int_0^T\int_\Omega
\left(
|e_v|^2
+
|\nabla e_v|^2
+
|\nabla(\nabla\cdot e_v)|^2
\right)\rho_{\theta_S}\,dxdt ,
\end{equation}
where $C_1(T):=C_1 e^{C_1T}$.
Next we estimate the term involving $\|e_\rho\|^2_{L^2}$. From the continuity equation for both $\rho_{\theta_S}$ and $\rho^*$, 
we obtain 
\[
\partial_t e_\rho
+
\nabla\cdot(e_\rho v_{\theta_S})
+
\nabla\cdot(\rho^*e_v)
=0.
\]
Multiplying by \(e_\rho\) and integrating over \(\Omega\) gives 
\[
\int_\Omega e_\rho \partial_t e_\rho\,dx
+
\int_\Omega e_\rho
\nabla\cdot(e_\rho v_{\theta_S})\,dx
+
\int_\Omega e_\rho
\nabla\cdot(\rho^*e_v)\,dx
=0.
\]
Therefore,
\[
\frac12\frac{d}{dt}\|e_\rho\|_{L^2}^2
=
-\int_\Omega e_\rho
\nabla\cdot(e_\rho v_{\theta_S})\,dx
-\int_\Omega e_\rho
\nabla\cdot(\rho^*e_v)\,dx .
\]
For the first term, integration by parts twice yields 
\[
\begin{aligned}
-\int_\Omega e_\rho
\nabla\cdot(e_\rho v_{\theta_S})\,dx
&=
 \frac12
\int_\Omega
v_{\theta_S} \cdot \nabla|e_\rho|^2dx
&=
-\frac12
\int_\Omega
(\nabla\cdot v_{\theta_S})
e_\rho^2\,dx,
\end{aligned}
\]
and since $\rho_{\theta_S} > 0$ and $\rho_{\theta_S}v_{\theta_S} \cdot \mathbf{n} = 0$ on $\partial \Omega$, the boundary term vanishes.

Hence,
\[
\left|
-\int_\Omega e_\rho
\nabla\cdot(e_\rho v_{\theta_S})\,dx
\right|
\le
\frac12
\|\nabla \cdot v_{\theta_S}\|_{L^\infty}
\|e_\rho\|_{L^2}^2 .
\]
For the second term, expanding the divergence gives
\[
\begin{aligned}
\left|
\int_\Omega e_\rho
\nabla\cdot(\rho^*e_v)\,dx
\right|
&=
\left|
\int_\Omega
e_\rho
\left(
\nabla\rho^*\cdot e_v
+
\rho^* \nabla\cdot e_v
\right)
dx
\right| \\
&\le
\int_\Omega
|e_\rho|
\Big(
|\nabla\rho^*||e_v|
+
|\rho^*||\nabla\cdot e_v|
\Big)
dx.
\end{aligned}
\]
Applying Young’s inequality and the Cauchy–Schwarz inequality, we obtain 
\[
\begin{aligned}
\left|
\int_\Omega e_\rho
\nabla\cdot(\rho^*e_v)\,dx
\right|
\le
\|e_\rho\|_{L^2}^2 +
\frac12 \|\nabla\rho^*\|_{L^\infty}^2
\|e_v\|_{L^2}^2+ 
\frac12 \|\rho^*\|_{L^\infty}^2
\|\nabla\cdot e_v\|_{L^2}^2 .
\end{aligned}
\]
Combining the above estimates yields
\[
\frac{d}{dt}\|e_\rho\|_{L^2}^2
\le
C_2(\|e_\rho\|_{L^2}^2
+
\|e_v\|_{L^2}^2
+
\|\nabla\cdot e_v\|_{L^2}^2),
\]
where $$
C_2 = \max\Big\{ C_v + 2, \|\nabla_x \rho^*\|_{L^{\infty}}^2, \|\rho^* \|_{L^{\infty}}^2 \Big\}.
$$
Applying Gronwall's inequality, for \(0\le t\le T\), with \(e_\rho(0)=0\), we obtain 
\begin{equation}\label{erho}
\begin{aligned}
\|e_\rho(t,\cdot)\|_{L^2(\Omega)}^2
& \le
C_2(T) \int_0^T
\left(
\|e_v(t)\|_{L^2}^2
+
\|\nabla\cdot e_v(t)\|_{L^2}^2
\right)
dt.
\end{aligned}
\end{equation}
Combining (\ref{es}) and (\ref{erho})  with \eqref{LD}, we obtain
\begin{equation*}
L_S(\theta_S)
\le
K(\theta_S) \int_0^T 
(\|e_v(t,\cdot)\|^2_{L^\infty}
+\|\nabla_x e_v(t,\cdot)\|^2_{L^\infty}
+\|\nabla_x(\nabla_x\cdot e_v(t, \cdot)\|
^2_{L^\infty})dt,
\end{equation*}
where 
$$
K(\theta_S)=
3\max \left\{ \max_{0 \le t\le T}\|\rho_{\theta_S}\|_{L^1(\Omega)}, 
C_1(T) \max_{0 \le t\le T}\|\rho_{\theta_S}\|_{L^1(\Omega)}, 
|\Omega|C_2(T)\|\nabla W\|^2_{L^1(\Omega)}  \|\rho_{\theta_S}\|_{L^\infty(\Omega_T)} \right\} 
$$
It remains to show that $K(\theta_S)$ is uniformly bounded in terms of $\theta_S$. 
In fact, using the mass conservation we have 
$$
\|\rho_{\theta_S}\|_{L^1(\Omega)}=M. 
$$
On the other hand, along the particle trajectories generated by $v_{\theta_S}$, the density satisfies 
\begin{equation*}
\rho_{\theta_S}(t,X_t)
=
\rho_0(X_0)
\exp\left(
-\int_0^t
\nabla_x\cdot v_{\theta_S}(s,X_s)\,ds
\right).
\end{equation*}
Consequently,
\begin{equation*}
\|\rho_{\theta_S}(t)\|_{L^\infty(\Omega_T)}
\le
\|\rho_0\|_{L^\infty(\Omega)}
\exp\left(
T\|\nabla_x\cdot v_{\theta_S}\|_{L^\infty(\Omega_T)}
\right)
\end{equation*}
By the assumptions of Theorem \ref{Thm_total_error}, since $\|\nabla_x \cdot v_{\theta_S}\|_{L^\infty(\Omega_T)}$ is bounded, so there exists a constant $\bar C>0$ depending on $\|\rho_0\|_{L^\infty}$ and $T$, such that
\begin{equation}\label{nabla_v}
\|\rho_{\theta_S}(t)\|_{L^\infty(\Omega_T)}
\le \bar C .
\end{equation}
This implies that  
$$
K(\theta_S) \leq C_S:=  3\max \left\{M, 
MC_1(T), 
|\Omega|C_2(T)\bar C \|\nabla W\|^2_{L^1(\Omega} \right\}.
$$

Therefore, we have
\begin{equation*}
    L_S(\theta_S) \leq C_S
\int_0^T
\|v_{\theta_S}(t,\cdot)-v^*(t,\cdot)\|^2_{W^{2,\infty}(\Omega)}
\,dt.
\end{equation*}

Before proving Lemmas \ref{Lemma_MC}, we first establish two auxiliary lemmas.
\begin{lemma}[Pushforward property of the characteristic flow]\label{lem:pushforward}
Let $x^{(n)}(t)$ denote the exact solution of
\begin{equation*}
\dot{x}^{(n)}(t)
=
v_\theta(t,x^{(n)}(t)),
\qquad
x^{(n)}(t_0)=x_0^{(n)},
\end{equation*}
and define 
\begin{equation*}
s^{(n)}(t)
:=
\nabla_x\log\rho_\theta(t,x^{(n)}(t)).
\end{equation*}
If $x_0^{(1)},\ldots,x_0^{(N_p)}$ are i.i.d. samples from
$\rho_0/M$, then, for every $t_j$, the random variables 
$$x^{(1)}(t_j),\ldots,x^{(N_p)}(t_j)$$ are i.i.d. samples from
$\rho_\theta(t_j,\cdot)/M$.
\end{lemma}

\begin{proof}
Since $\rho_\theta$ satisfies the continuity equation 
\begin{equation*}
\partial_t\rho_\theta
+
\nabla\cdot(\rho_\theta v_\theta)
=
0,
\end{equation*}
the flow map $x_0\mapsto x^{(n)}(t_j)$ 
generated by $v_\theta$ transports the initial density 
$\rho_0/M$ to $\rho_\theta(t_j,\cdot)/M$.
Thus, each $x^{(n)}(t_j)$ has distribution $\rho_\theta(t_j, \cdot)/M$.  

Moreover,  $x^{(n)}(t_j)$ depends only on the corresponding initial sample $x_0^{(n)}$. Since the initial samples are independent, independence is preserved under the characteristic flow. 
\end{proof}

\begin{lemma}[Monte Carlo error, high-probability estimate]\label{lem:mc_error}
Let
\begin{equation*}
g_\theta(t,x)
:=
\left|
v_\theta(t,x)
+
s_\theta(t,x)
+
\nabla V(x)
+
F_\theta(t,x)
\right|^2,
\end{equation*}
where
\begin{equation*}
F_\theta(t,x)
=
\nabla_x(W*\rho_\theta)(t,x),
\end{equation*}
and define
\begin{equation*}
\Phi(t_j)
=
\int_\Omega
g_\theta(t_j,x)\rho_\theta(t_j,x)\,dx.
\end{equation*}
Under the assumptions of Theorem \ref{Thm_total_error}, $g_\theta$ is
uniformly bounded on $[0,T]\times\Omega$: there exists $G>0$ such that
$0\le g_\theta(t,x)\le G$ for all $(t,x)\in[0,T]\times\Omega$.

Suppose that  $x^{(1)}(t_j),\ldots,x^{(N_p)}(t_j)$ are i.i.d.\ samples from
$\rho_\theta(t_j,\cdot)/M$ as established in Lemma \ref{lem:pushforward}. Then, for any
$\delta\in(0,1)$, with probability at least $1-\delta/2$,
\begin{equation*}
\left|
\sum_{j=0}^{N_t-1}
\left(
\frac{M}{N_p}
\sum_{n=1}^{N_p}
g_\theta(t_j,x^{(n)}(t_j))
-
\Phi(t_j)
\right)
\Delta t
\right|
\le
MGT
\sqrt{\frac{\log(4N_t/\delta)}{2N_p}}.
\end{equation*}
In particular, the Monte Carlo quadrature error is of order $N_p^{-1/2}$, up to a logarithmic factor.

If, in addition, the interaction force is approximated by the empirical force 
\begin{equation*}
\hat F^{(n)}(t_j)
=
\frac{M}{N_p}
\sum_{m=1}^{N_p}
\nabla W
\left(
x^{(n)}(t_j)-x^{(m)}(t_j)
\right).
\end{equation*}
Accordingly, define the empirical loss integrand by
\begin{equation*}
\hat g_\theta(t_j,x^{(n)}(t_j))
:=
\left|
v_\theta(t_j,x^{(n)}(t_j))
+
s_\theta(t_j,x^{(n)}(t_j))
+
\nabla V(x^{(n)}(t_j))
+
\hat F^{(n)}(t_j)
\right|^2.
\end{equation*}
Then, for any $\delta \in (0, 1)$, with probability at least $1-\delta/2$, the force approximation satisfies 
\begin{equation*}
\max_{n, j} \left|
\hat F^{(n)}(t_j)
-
F_\theta
\left(
t_j,x^{(n)}(t_j)
\right)
\right|
\le
C_F \sqrt{\frac{\log(4N_pN_t/\delta)}{N_p}},
\end{equation*}
where $C_F>0$ depends only on $\|\nabla W\|_{L^\infty(\Omega)}$.

Consequently, with probability at least $1-\delta$,
\begin{equation*}
\left|
\sum_{j=0}^{N_t-1}
\left(
\frac{M}{N_p}
\sum_{n=1}^{N_p}
\hat g_\theta(t_j,x^{(n)}(t_j))
-
\Phi(t_j)
\right)
\Delta t
\right|
\le
C_{\mathrm{MC}}
\sqrt{\frac{\log(4N_pN_t/\delta)}{N_p}},
\end{equation*}
where $C_{\mathrm{MC}} > 0$ is independent of $N_p$ and $N_t$.
\end{lemma}

\begin{proof}
Fix $j\in \{0, \cdots, N_t-1\}$. By Lemma \ref{lem:pushforward}, the random variables $x^{(1)}(t_j),\ldots,x^{(N_p)}(t_j)$ are i.i.d.\ samples from
$\rho_\theta(t_j,\cdot)/M$. Since  $g_\theta(t_j,\cdot)\in[0,G]$, the random variables 
 $Y_n:=g_\theta(t_j,x^{(n)}(t_j))$ are i.i.d and satisfy 
 $0\leq Y_n \leq G$. Moreover, 
 $$
 \mathbb{E}[Y_n]=\Phi(t_j)/M. 
 $$
 Hence, by Hoeffding's
inequality, for any $\varepsilon>0$,
\begin{equation*}
\mathbb{P}
\left(
\left|
\frac{M}{N_p}\sum_{n=1}^{N_p}g_\theta(t_j,x^{(n)}(t_j))-\Phi(t_j)
\right|
\ge
\varepsilon
\right)
\le
2\exp\left(-\frac{2N_p\varepsilon^2}{M^2G^2}\right).
\end{equation*}
Choosing 
$$
\varepsilon= MG\sqrt{\frac{\log(4N_t/\delta)}{2N_p}}, 
$$
the right-hand side becomes $\delta/(2N_t)$. Therefore,  
by a union bound over $j=0,\ldots,N_t-1$, with
probability at least $1-\delta/2$, 
\begin{equation}\label{eq:g_bound}
\left|
\frac{M}{N_p}\sum_{n=1}^{N_p}g_\theta(t_j,x^{(n)}(t_j))-\Phi(t_j)
\right|
\le
MG\sqrt{\frac{\log(4N_t/\delta)}{2N_p}}
\end{equation}
simultaneously for all $j$. Multiplying by $\delta t$, summing over $j$, and using $N_t \Delta t=T$ gives
the first claimed estimate.

We next estimate the empirical interaction force. Fix $n$ and $j$, and condition on
$x^{(n)}(t_j)$. Since $W(x) = W(-x)$, $\nabla W$ is odd, so $\nabla W(0) = 0$; hence the
term $m = n$ in
\[
\hat F^{(n)}(t_j) = \frac{M}{N_p} \sum_{m=1}^{N_p} \nabla W\big(x^{(n)}(t_j) - x^{(m)}(t_j)\big)
\]
vanishes, and consequently
\[
\hat F^{(n)}(t_j) = \frac{N_p - 1}{N_p} \, \widetilde F^{(n)}(t_j),
\qquad
\widetilde F^{(n)}(t_j) := \frac{M}{N_p - 1} \sum_{m \neq n} \nabla W\big(x^{(n)}(t_j) - x^{(m)}(t_j)\big).
\]
For $m \neq n$, the random variables $\nabla W\big(x^{(n)}(t_j) - x^{(m)}(t_j)\big)$ are,
by Lemma \ref{lem:pushforward}, conditionally independent and identically distributed, with
conditional expectation $\frac{1}{M} F_\theta(t_j, x^{(n)}(t_j))$. Thus $\widetilde F^{(n)}(t_j)$
is conditionally unbiased for $F_\theta(t_j, x^{(n)}(t_j))$. Since $\nabla W \in W^{1,\infty}(\Omega)$,
applying Hoeffding's inequality componentwise gives
\begin{equation*}
\begin{aligned}
    \mathbb{P}\Bigg(
\big| \widetilde F^{(n)}(t_j) - F_\theta(t_j, x^{(n)}(t_j)) \big| \geq \eta
\;\Bigg|\; x^{(n)}(t_j)
\Bigg) \notag 
 \leq 2 \exp\!\big(-C_1 (N_p - 1) \eta^2 / \|\nabla W\|_{L^\infty}^2\big),
\end{aligned}
\end{equation*}
where $C_1$ is universal positive constants. A union bound over
$n = 1, \dots, N_p$ and $j = 0, \dots, N_t - 1$ then shows that, with probability at least
$1 - \delta/2$,
\[
\max_{n,j}
\big| \widetilde F^{(n)}(t_j) - F_\theta(t_j, x^{(n)}(t_j)) \big|
\leq C\, \|\nabla W\|_{L^\infty} \sqrt{\frac{\log(4 N_p N_t / \delta)}{N_p - 1}},
\]
where $C>0$ is a universal positive constant. Combining this with
\begin{equation*}
\hat F^{(n)}(t_j) - F_\theta(t_j, x^{(n)}(t_j))
= \frac{N_p - 1}{N_p}
\Big[ \widetilde F^{(n)}(t_j) - F_\theta(t_j, x^{(n)}(t_j)) \Big]
- \frac{1}{N_p} F_\theta(t_j, x^{(n)}(t_j)),
\end{equation*}
and using $\frac{N_p-1}{N_p} \in (0,1)$ together with the triangle inequality, we obtain
\begin{equation*}
\big| \hat F^{(n)}(t_j) - F_\theta(t_j, x^{(n)}(t_j)) \big|
\leq
\big| \widetilde F^{(n)}(t_j) - F_\theta(t_j, x^{(n)}(t_j)) \big|
+
\frac{1}{N_p} \big| F_\theta(t_j, x^{(n)}(t_j)) \big|.
\end{equation*}
By the boundedness $|F_\theta| \leq C_F$ from Assumption \ref{Ass_thm}, and the estimate for
$\widetilde F^{(n)}(t_j)$ established above, we obtain that, with probability at least
$1 - \delta/2$,
\begin{equation*}
\max_{n,j}
\big| \hat F^{(n)}(t_j) - F_\theta(t_j, x^{(n)}(t_j)) \big|
\leq
C \|\nabla W\|_{L^\infty} \sqrt{\frac{\log(4 N_p N_t / \delta)}{N_p - 1}}
+
\frac{C_F}{N_p}.
\end{equation*}
With fixed $N_t$, $\delta \in (0,1)$, for  $N_p$ large enough, we conclude that with probability at least
$1-\delta/2$,
\begin{equation}\label{eq:F_bound}
\max_{n,j}
\left|
\hat F^{(n)}(t_j) - F_\theta\big(t_j, x^{(n)}(t_j)\big)
\right|
\leq C_F \sqrt{\frac{\log(4 N_p N_t / \delta)}{N_p}},
\end{equation}
where $C_F>0$ depends only on $\|\nabla W\|_{L^\infty(\Omega)}$. 
By the uniform boundedness assumptions,
\begin{equation*}
\left|
\hat g_\theta(t_j,x^{(n)}(t_j))
-
g_\theta(t_j,x^{(n)}(t_j))
\right|
\le
C_g
\left|
\hat F^{(n)}(t_j)
-
F_\theta(t_j,x^{(n)}(t_j))
\right|,
\end{equation*}
where $C_g > 0$ depends only on the uniform bounds in the assumptions. Combining this estimate with \eqref{eq:F_bound}, we obtain
\begin{equation*}
\max_{n,j}
\left|
\hat g_\theta(t_j,x^{(n)}(t_j))
-
g_\theta(t_j,x^{(n)}(t_j))
\right|
\le
C_g C_F
\sqrt{\frac{\log(4N_pN_t/\delta)}{N_p}}.
\end{equation*}

By a union bound, the two
estimates above hold simultaneously with probability at least $1-\delta$. In particular, by the triangle inequality and 
using $\log(4N_t/\delta)\le\log(4N_pN_t/\delta)$, with probability at least $1-\delta$,
\begin{equation*}
\left|
\sum_{j=0}^{N_t-1}
\left(
\frac{M}{N_p}
\sum_{n=1}^{N_p}
\hat g_\theta(t_j,x^{(n)}(t_j))
-
\Phi(t_j)
\right)
\Delta t
\right|
\le
C_{\mathrm{MC}}
\sqrt{\frac{\log(4N_pN_t/\delta)}{N_p}},
\end{equation*}
where $C_{\mathrm{MC}} = MGT + MTC_g C_F$, independent of $N_p$ and $N_t$.

\end{proof}

\begin{remark}[Empirical interaction and self-interaction]
The interaction term takes the form of the empirical average $\frac{M}{N_p}\sum_{j\neq i}^{N_p}\nabla W(X_i-X_j)$,
where the self-interaction contribution $j=i$ is set to zero.
For a finite interacting particle system, the particle trajectories are
generally correlated through the empirical mean-field interaction and
therefore cannot be treated as independent samples. Quantitative
mean-field and propagation-of-chaos results provide a rigorous framework
for controlling this dependence and for justifying the convergence of
the empirical interaction force toward its mean-field counterpart
$(\nabla W * \rho)(X_i)$; see, for example,
\cite{JabinWang2016,BermanOnnheim2019}. Under the boundedness and
regularity assumptions on the interaction kernel adopted in the present
analysis, the corresponding empirical fluctuation is of the standard
order $N_p^{-1/2}$.
\end{remark}

\begin{lemma}[Euler consistency]\label{lem:euler_consistency}
Let $x_{t_j}^{(n)}, s_{t_j}^{(n)}$ be generated by the Euler scheme
\begin{align*}
& x^{(n)}_{t_{j+1}}=x^{(n)}_{t_j}+\Delta t\,v_\theta(t_j,x^{(n)}_{t_j}), \\
& s^{(n)}_{t_{j+1}}=s^{(n)}_{t_j}-\Delta t\left[(\nabla_x v_\theta(t_j,x_{t_j}^{(n)}))^\top s_{t_j}^{(n)}+\nabla_x(\nabla_x\cdot v_\theta(t_j,x_{t_j}^{(n)}))\right],
\end{align*}
with initial conditions 
$$
x_{t_0}^{(n)}=x_0^{(n)}, \quad s_{t_0}^{(n)}=\nabla_x\log\rho_0(x_0^{(n)}).
$$
Let $x^{(n)}(t_j), s^{(n)}(t_j)$ be the corresponding exact solutions, as in Lemma \ref{lem:pushforward}. Assume that $v_\theta$, $\nabla_x v_\theta$, and $\nabla_x(\nabla_x\cdot v_\theta)$ are uniformly bounded on $[0,T]\times\Omega$. Assume, in addition, that $v_\theta$ is Lipschitz continuous in $t$, uniformly in $x\in\Omega$; (i.e. $\partial_t v_\theta\in L^\infty([0,T]\times\Omega)$). We also assume the spatial regularity needed to ensure that the right-hand side of the coupled system is Lipschitz continuous in $(x,s)$. Then, for $0\le j\le N_t$,
\begin{equation*}
\bigl|x_{t_j}^{(n)}-x^{(n)}(t_j)\bigr|+\bigl|s_{t_j}^{(n)}-s^{(n)}(t_j)\bigr|\ \le\ C\Delta t,
\end{equation*}
where $C$ is independent of $\Delta t$, $N_p$ and $N_t$. 
\end{lemma}
\begin{proof} For fixed $n$, write $z=(x,s)$ and formulate the coupled characteristic equations as 
$$
\dot z=f(t, z), 
$$
where 
$$
f(t, z)=(v_\theta(t, x), -(\nabla_x v_\theta(t, x))^\top s -\nabla(\nabla_x\cdot v_\theta(t, x))).
$$
By the assumed regularity and uniform bounds, $f$ is Lipschitz continuous in $z$, uniformly for $t\in [0,T]$. The assumed time regularity of $v_\theta$, together with the corresponding regularity of the coefficients in the s-equation, also ensures that $f$ is Lipschitz continuous in $t$ along the exact trajectory.  Consequently, the local truncation error of the Euler scheme is of order $O(\Delta t^2)$. More precisely, for each time step,
$$
z(t_{j+1})=z(t_j) +\Delta t f(t_j, z(t_j))+R_j, 
$$
where $|R_j| \leq C_0 (\Delta t)^2$, where $C_0$ is independent of $\Delta t$ and $N_p$. 

Define the global error by 
$$
e_j=z^{(n)}(t_j)-z^{(n)}(t_j).
$$
Since the numerical and exact solutions have the same initial data, $e_0=0$. Using the Euler scheme, the above expansion of the exact solution, and the Lipschitz continuity of $f$ in $z$, we obtain
$$
|e_{j+1}| \leq (1+L\Delta t)|e_j|+C_0(\Delta t)^2,
$$
where $L$ is a Lipschitz constant for $f$. 

Applying the discrete Gr\"{o}nwall inequality gives
$$
|e_j| \leq  C \Delta t,
$$
for all $0 \leq  j \leq  N_t$, where $C$ depends on $T$ and the uniform bounds and regularity constants entering the coupled system, but is independent of $\Delta t, N_t$, and $N_p$.

Since $e_j$ consists of the position and score errors, the desired estimate follows.
\end{proof}

Proof of Lemma \ref{Lemma_MC}.

For the analysis, we introduce the intermediate particle loss $\hat L_S$, evaluated along the exact characteristic trajectories. This quantity provides a bridge between the population loss $L_S$ and the fully discrete loss $L_D$. It is defined by
\begin{equation*}
\hat L_S(\theta)
=
\frac{M}{N_p}
\sum_{j=0}^{N_t-1}
\sum_{n=1}^{N_p}
\left|
v_\theta(t_j,x^{(n)}(t_j))
+
s^{(n)}(t_j)
+
\nabla V(x^{(n)}(t_j))
+
\hat F^{(n)}(t_j)
\right|^2
\Delta t,
\end{equation*}
where
\begin{equation*}
\hat F^{(n)}(t_j)
=
\frac{M}{N_p}
\sum_{m=1}^{N_p}
\nabla W\bigl(
x^{(n)}(t_j)-x^{(m)}(t_j)
\bigr).
\end{equation*}
Here, $x^{(n)}(t_j)$ and $s^{(n)}(t_j)$ are evaluated along the exact characteristic trajectories.\\

Recall the decomposition in \eqref{L_Decomp_1}. For any fixed $\theta$, we decompose the discrepancy between the population loss and the fully discrete loss through the intermediate loss $\hat L_S$ as  
\begin{equation*}
    L_S(\theta) - L_D(\theta) = \Big[L_S(\theta) - \hat L_S (\theta)\Big] + \Big[\hat L_S(\theta) - L_D (\theta)\Big].
\end{equation*}
Substituting this decomposition into
\eqref{L_Decomp_1} gives 
\begin{equation}\label{L_Decomp}
\begin{aligned}
L_S(\theta_k) & \leq
L_S(\theta_S)
+
\Bigl[
L_D(\theta_k)-L_D(\theta_D)
\Bigr]
 \\
&\quad +
\Big[L_S(\theta_k) - \hat L_S (\theta_k)\Big] + \Big[\hat{L}_S(\theta_k) - L_D (\theta_k)\Big]\\
&\quad + 
\Big[\hat L_S(\theta_S) - L_S (\theta_S)\Big] + \Big[L_D (\theta_S) - \hat{L}_S(\theta_S)\Big].
\end{aligned}
\end{equation}
Thus, the total error is decomposed into the optimization error, the sampling error, and the time-discretization error.
We next estimate the two approximation terms.

For any fixed $\theta$, define
the time-discrete population loss 
\begin{equation*}
\hat L_D(\theta)
=
\sum_{j=0}^{N_t-1}
\int_\Omega
\left|
v_\theta(t_j,x)
+s_\theta(t_j,x)
+\nabla V(x)
+F_\theta(t_j,x)
\right|^2
\rho_\theta(t_j,x)\,dx\,\Delta t,
\end{equation*}
where
\begin{equation*}
F_\theta(t_j,x)
=
\nabla_x(W * \rho_\theta)(t_j,x).
\end{equation*}
Here, $\hat L_D(\theta)$ retains the exact spatial expectation with respect to $\rho_\theta(t_j, \cdot)$, while replacing the time integral by the left-point rectangle rule.

We first write 
\begin{equation*}
L_S(\theta)-\hat L_S(\theta)
=
\bigl[L_S(\theta)-\hat L_D(\theta)\bigr]
+
\bigl[\hat L_D(\theta)-\hat L_S(\theta)\bigr].
\end{equation*}

The first term is the time discretization error, arising from the
rectangle rule approximating
$\int_0^T \Phi(t)\,dt$ by
$\sum_{j=0}^{N_t-1}\Phi(t_j)\Delta t$.
Under the assumptions of Theorem \ref{Thm_total_error}, the integrand is sufficiently regular in time, and the standard rectangle-rule estimate yields
\begin{equation}\label{eq:time_discretization}
\left|
L_S(\theta)-\hat L_D(\theta)
\right|
\le C_D\Delta t.
\end{equation}

The second term is the Monte Carlo approximation error.
By Lemma \ref{lem:pushforward}, the exact-flow particles
$x^{(n)}(t_j)$ are i.i.d. samples from
$\rho_\theta(t_j,\cdot)/M$. Applying Lemma \ref{lem:mc_error} to both the outer spatial
expectation and the interaction-force approximation, we obtain with probability at least
$1-\delta$,
\begin{equation}\label{eq:MC_error}
\left|
\hat L_D(\theta)-\hat L_S(\theta)
\right|
\le C_{\mathrm{MC}}\sqrt{\frac{\log(N_pN_t/\delta)}{N_p}},
\end{equation}
where $C_{\mathrm{MC}}$ is the constant given in Lemma \ref{lem:mc_error}.
Combining \eqref{eq:time_discretization} and
\eqref{eq:MC_error}, we obtain
\begin{equation}\label{eq:LS_hatLS_error}
\left|
L_S(\theta)-\hat L_S(\theta)
\right|
\le
C_{\mathrm{MC}}\sqrt{\frac{\log(N_pN_t/\delta)}{N_p}}
+
C_D\Delta t.
\end{equation}
We next estimate the discrepancy between the exact-flow particle loss and the fully
discrete loss. The latter is evaluated along the Euler approximations $x^{(n)}_{t_j}$ and
$s^{(n)}_{t_j}$, rather than the exact characteristic trajectories $x^{(n)}(t_j)$ and
$s^{(n)}(t_j)$. By Lemma \ref{lem:euler_consistency}, these errors are of order
$O(\Delta t)$, uniform in $n$ and $j$. Under the assumed uniform regularity and
boundedness of the loss integrand, this implies
\begin{equation}\label{eq:euler_error}
\left|
\hat L_S(\theta)-L_D(\theta)
\right|
\le C_E\Delta t,
\end{equation}
where $C_E$ is independent of $N_p$ and $\Delta t$.
We now apply these estimates to $\theta_k$ and $\theta_S$. Since $\theta_k$ is evaluated
on a batch independent of the training data, the preceding high-probability estimates
apply to $\theta_k$ in the same manner as to the fixed parameter $\theta_S$. Using a union
bound, we may therefore choose the confidence levels so that, with probability at least
$1 - \delta$, all four approximation terms appearing in the decomposition above satisfy
the corresponding estimates simultaneously.
Hence, by the triangle inequality
\begin{equation*}
\begin{aligned}
&
\bigl[L_S(\theta_k)-\hat L_S(\theta_k)\bigr]
+\bigl[\hat L_S(\theta_k)-L_D(\theta_k)\bigr]
+ \Big[\hat L_S(\theta_S) - L_S (\theta_S)\Big]
+ \Big[L_D (\theta_S) - \hat{L}_S(\theta_S)\Big]
\\
&\qquad\le
2 C_{\mathrm{MC}}\sqrt{\frac{\log(N_pN_t/\delta)}{N_p}}
+
2(C_D+C_E)\Delta t
\\
&\qquad\le
 C_{\mathrm{MD}}
\left(
\sqrt{\frac{\log(N_pN_t/\delta)}{N_p}}+\Delta t
\right),
\end{aligned}
\end{equation*}
where
\begin{equation*}
 C_{\mathrm{MD}}
=
2\max\left\{
 C_{\mathrm{MC}},\,C_D+C_E
\right\}.
\end{equation*}
The constant $C_{\mathrm{MD}}$ is independent of $N_p$ and $\Delta t$.
This completes the proof.

\section{Multi-stage Training Strategy}\label{AppendixB}
\subsection{Multi-stage training for long-time simulations}

For long-time simulations, training a single neural network over the entire time interval $[0, T]$ can become increasingly difficult and may lead to reduced efficiency and accuracy. To address this issue, we adopt a multi-stage training strategy following \cite{zhou2025}. Specifically, we partition $[0,T]$ into $N_{\mathrm{stage}}$ consecutive subintervals,
$$
0=T_0<T_1<\cdots<T_{N_{\mathrm{stage}}}=T.
$$
At stage $k$, the neural velocity field is trained only on the subinterval $[T_k,T_{k+1}]$. The particle positions and auxiliary variables obtained at time $T_k$ are used as the initial data for this stage. Once the network is trained, the resulting flow is propagated from
$T_k$ to $T_{k+1}$, and the resulting particle states and auxiliary variables are then used to initialize the next stage.

To further improve optimization efficiency, the network parameters at each stage are initialized using the optimized parameters from the preceding stage. This strategy provides a natural continuation of the learned velocity field across successive time intervals and reduces the computational cost of retraining from scratch.

Overall, the multi-stage strategy decomposes the long-time simulation into a sequence of shorter, more manageable learning problems while maintaining continuity of the particle evolution. This improves the robustness and efficiency of the proposed method for long-time simulations. The complete procedure is summarized in Algorithm~\ref{alg:multi_stage}.

\subsection{Multi-stage Algorithm}
\leavevmode
\vskip -1em
\begin{breakablealgorithm}
\caption{Generalized score-based neural transport for long time horizon}
\label{alg:multi_stage}
\begin{algorithmic}[1]

\Require Time horizon $T$, number of stages $N_{\mathrm{stage}}$, time steps $N_t$, particle number $N_p$, initial density $\rho_0$, learning rate, optimization steps per stage $N_{\mathrm{step}}$

\State Divide the time interval $[0,T]$ into $N_{\mathrm{stage}}$ stages:
\State $\displaystyle 0=T_0<T_1<\cdots<T_{N_{\mathrm{stage}}}=T$

\State Initialize parameters $\theta$ of $v_\theta$

\For{$k=0,1,\cdots,N_{\mathrm{stage}}-1$}

\State Set the current training interval $[T_k,T_{k+1}]$

\For{$\mathrm{step}=1,2,\cdots,N_{\mathrm{step}}$}

\State Sample particles $\{x_{T_k}^{(n)}\}_{n=1}^{N_p}$ from the density at time $T_k$

\State Initialize density variables $\rho_{T_k}^{(n)}$ and auxiliary variables $q_{T_k}^{(n)}$

\State Set $\mathrm{Loss}=0$

\For{$t_j\in [T_k,T_{k+1})$}

\State Evaluate the neural velocity $v_\theta(t_j,x_{t_j}^{(n)})$

\State Compute the interaction force

\State $\displaystyle
F_{t_j}^{(n)}
=
\frac{M}{N_p}
\sum_{m=1}^{N_p}
\nabla W(x_{t_j}^{(n)}-x_{t_j}^{(m)})
$

\State Accumulate objective

\State $\displaystyle
\mathrm{Loss}
=
\mathrm{Loss}
+
\frac{M}{N_p}
\sum_{n=1}^{N_p}
\left|
v_\theta(t_j,x_{t_j}^{(n)})
+
q_{t_j}^{(n)}
+
\nabla V(x_{t_j}^{(n)})
+
F_{t_j}^{(n)}
\right|^2
\Delta t
$

\State Update particle position

\State $\displaystyle
x_{t_{j+1}}^{(n)}
=
x_{t_j}^{(n)}
+
\Delta t\,
v_\theta(t_j,x_{t_j}^{(n)})
$

\State Update density variable

\State $\displaystyle
\rho_{t_{j+1}}^{(n)}
=
\rho_{t_j}^{(n)}
-
\Delta t\,
\rho_{t_j}^{(n)}
\nabla_x\cdot
v_\theta(t_j,x_{t_j}^{(n)})
$

\State Update auxiliary variable

\State $\displaystyle
\begin{aligned}
q_{t_{j+1}}^{(n)}
={}&
q_{t_j}^{(n)}
-\Delta t
\Bigg[
(\nabla_x v_\theta(t_j,x_{t_j}^{(n)}))^\top q_{t_j}^{(n)}
+\rho_{t_j}^{(n)}H''(\rho_{t_j}^{(n)})
\nabla_x(\nabla_x\cdot v_\theta(t_j,x_{t_j}^{(n)}))
\\
&+
(\nabla_x\cdot v_\theta(t_j,x_{t_j}^{(n)}))
\left(
1+
\frac{\rho_{t_j}^{(n)}H'''(\rho_{t_j}^{(n)})}
{H''(\rho_{t_j}^{(n)})}
\right)
q_{t_j}^{(n)}
\Bigg].
\end{aligned}
$

\EndFor

\State Update $\theta$ with Adam to minimize $\mathrm{Loss}$

\EndFor

\State Use the trained flow on $[T_k,T_{k+1}]$ to propagate particles to $T_{k+1}$

\State Set the propagated particles, density variables, and auxiliary
variables as the initial data for the next stage, and warm-start the
network with the optimized parameters from the current stage

\EndFor

\end{algorithmic}
\end{breakablealgorithm}

\noindent {\bf Implementation details.} 
\begin{itemize}
\item {\bf Architecture.}  Fully connected neural network, 4 hidden
layers, 128 neurons per layer, $\tanh$ activation.
\item {\bf Optimization:} Adam optimizer ($lr=10^{-3}$, $N_p=512$ particles/patch, $2000$ iterations/stage,no
early stopping).
\item {\bf Continuation:}  Successive stages $[T_k,T_{k+1}]$ are warm-started by continuing both network parameters and states variables (particles, density, generalized score fields) at $T_{k+1}$.  
\end{itemize}

\end{document}